\documentclass[12pt]{article}
\usepackage{amsmath,amsfonts,amssymb,amsthm}
\begin{document}

\newtheorem{thm}{Theorem}[section]
\newtheorem{prop}[thm]{Proposition}
\newtheorem{coro}[thm]{Corollary}
\newtheorem{conj}[thm]{Conjecture}
\newtheorem{example}[thm]{Example}
\newtheorem{lem}[thm]{Lemma}
\newtheorem{rem}[thm]{Remark}
\newtheorem{hy}[thm]{Hypothesis}
\newtheorem*{acks}{Acknowledgements}

\theoremstyle{definition}
\newtheorem{de}[thm]{Definition}

\newcommand{\C}{{\mathbb{C}}}
\newcommand{\Z}{{\mathbb{Z}}}
\newcommand{\N}{{\mathbb{N}}}
\newcommand{\Zr}{{\Z^r}}
\newcommand{\Zall}{{\Z^{r+1}}}
\newcommand{\Zalltwisted}{{\ovN\Z\times\Z^{r}}}

\newcommand{\Aut}{{{\rm Aut}}}
\newcommand{\End}{{{\rm End} }}
\newcommand{\Hom}{{\rm Hom}}
\newcommand{\Image}{{\rm Im}}
\newcommand{\Ker}{{{\rm Ker}}}
\newcommand{\Mod}{{{\rm Mod}}}
\newcommand{\Res}{{\rm Res}}
\newcommand{\Span}{{\rm Span}}

\newcommand{\rsymbol}[1]{{\bf {#1}}}

\newcommand{\bda}{{\rsymbol{a}}}
\newcommand{\x}{{\rsymbol{x}}}
\newcommand{\y}{{\rsymbol{y}}}
\newcommand{\z}{{\rsymbol{z}}}
\newcommand{\m}{{\rsymbol{m}}}
\newcommand{\Nfactor}{{\rsymbol{N}}}
\newcommand{\n}{{\rsymbol{n}}}
\newcommand{\bk}{{\rsymbol{k}}}
\newcommand{\bdt}{{\rsymbol{t}}}
\newcommand{\choice}[2]{{\left( \renewcommand{\arraystretch}{1.0} \begin{array}{c} {#1}\\{#2} \end{array} \right)}}

\newcommand{\dfunc}[3]{{#1^{-1}\delta\left(\frac{{#2}}{{#3}}\right)}}

\newcommand{\ESalltwisted}[1]{{\mathcal{E}\left({#1},r;N_{0}\right)}}
\newcommand{\ESalltwistedtemp}[1]{{\mathcal{E}\left({#1},r;1 \right)}}
\newcommand{\ESall}[1]{{\mathcal{E}\left({#1},r\right)}}

\newcommand{\ovN}{\frac{1}{N_{0}}}

\newcommand{\vac}{\rsymbol{1}}
\newcommand{\varz}[2]{{{#1}_0^{\pm {#2}}}}
\newcommand{\varr}[1]{{{#1}_1^{\pm 1},\dots,{#1}_r^{\pm 1}}}

\def \<{{\langle}}
\def \>{{\rangle}}

\newcommand{\tva}[1]{{toroidal vertex algebra}}
\newcommand{\rtva}[1]{{$(r+1)$-\tva{}}}
\newcommand{\Y}[3]{{Y\left( {#1};{#2},{#3} \right)}}
\newcommand{\Ymod}[4]{{Y_{#1}\left( {#2};{#3},{#4} \right)}}
\newcommand{\Ye}[3]{{\Ymod{{\mathcal{E}}}{#1}{#2}{#3}}}

\newcommand{\suball}[2]{{{#1},{#2}}}
\newcommand{\rangeall}[2]{{\left(\suball{{#1}}{{#2}}\right)\in \Zall}}
\newcommand{\rangealltwisted}[2]{{\left(\suball{{#1}}{{#2}}\right)\in \Zalltwisted}}
\newcommand{\sumall}[2]{{\sum\limits_{\rangeall{{#1}}{{#2}}}}}
\newcommand{\sumalltwisted}[2]{{\sum\limits_{\rangealltwisted{{#1}}{{#2}}}}}

\newcommand{\set}[2]{{\left\{ {#1} \,\left|\, {#2} \right.\right\} }}

\newcommand{\rtu}[1]{\omega_{#1}}

\renewcommand{\arraystretch}{1.5}

\newcommand{\stct}[1]{satisfying the condition that}
\newcommand{\andtext}{{\text{and}}}
\newcommand{\fortext}{{\text{for }}}

\newcommand{\cent}{{\mathfrak{c}}}
\newcommand{\g}{{\mathfrak{g}}}
\newcommand{\Lie}{{\tau}}
\newcommand{\Lieu}{{\mathcal{L}}}
\newcommand{\Borel}{{\mathfrak{B}}}
\newcommand{\Nil}{{\mathfrak{N}}}

\newcommand{\T}{{V_{\Lieu}\left( \ell,0 \right)}}

\makeatletter
\@addtoreset{equation}{section}
\def\theequation{\thesection.\arabic{equation}}
\makeatother
\makeatletter

\begin{center}
{\Large \bf Twisted Modules for Toroidal Vertex Algebras}
\end{center}

\begin{center}
{Fei Kong$^{a}$, Haisheng Li$^{b}$\footnote{Partially supported by
 China NSF grant (No. 11471268)},
Shaobin Tan$^{a}$\footnote{Partially supported by China NSF grant (No. 11471268)}
and Qing Wang$^{a}$\footnote{Partially supported by
 China NSF grant (No. 11371024), Natural Science Foundation of Fujian Province
(No. 2013J01018) and Fundamental Research Funds for the Central
University (No. 2013121001).}\\
$\mbox{}^{a}$School of Mathematical Sciences, Xiamen University,
Xiamen 361005, China\\
$\mbox{}^{b}$ Department of Mathematical Sciences,\\
Rutgers University, Camden, NJ 08102, USA\\}
\end{center}

\begin{abstract}
This is a paper in a series systematically to study toroidal vertex algebras.
Previously, a theory of toroidal vertex algebras and modules
was developed and toroidal vertex algebras were explicitly associated to toroidal Lie algebras.
In this paper, we study twisted modules for toroidal vertex algebras. More specifically, we introduce a notion of twisted module for a general toroidal vertex algebra with a finite order automorphism  and we give a general construction of toroidal vertex algebras and  twisted modules. We then use this construction to establish a natural association of toroidal vertex algebras and twisted modules to twisted toroidal Lie algebras. This together with some other known results
implies that almost all extended affine Lie algebras can be associated to toroidal vertex algebras.
\end{abstract}

\section{Introduction}
Extended affine Lie algebras are natural generalizations of affine Kac-Moody Lie algebras (see \cite{AABGP}), where
an important family of extended affine Lie algebras consists of  toroidal Lie algebras
 which are central extensions of multi-loop algebras of finite dimensional simple Lie algebras.
It is known that (untwisted) affine Lie algebras can be canonically associated with vertex algebras and modules (cf. \cite{FZ}, \cite{DL}, \cite{Li1}, \cite{LL}),
while twisted affine Lie algebras can be associated with vertex algebras in terms of twisted modules (see \cite{FLM}, \cite{Li2}).
In \cite{BBS},  a natural connection of certain toroidal Lie algebras with vertex algebras was established.
On the other hand, it is natural to study suitable toroidal analogues of vertex algebras and their relations with
toroidal Lie algebras from a different perspective.
This is also potentially important from the viewpoint of physics conformal field theory in high dimensions (cf. \cite{IKU}, \cite{IKUX}).
With this as the main driving force, a theory of toroidal vertex algebras and modules was developed
and toroidal vertex algebras and their modules were associated to toroidal Lie algebras in \cite{LTW}.

In the current paper, we continue to study twisted modules for toroidal vertex algebras, with a goal to associate twisted modules for certain toroidal vertex algebras to modules for twisted toroidal Lie algebras.
To achieve this goal, we develop a theory of twisted modules for a toroidal vertex algebra
with a finite order automorphism and we establish a conceptual construction of toroidal vertex algebras and twisted modules. By using this general result, we successfully associate twisted modules for certain  toroidal vertex algebras to
twisted toroidal Lie algebras.

Note that affine Kac-Moody algebras were classified as untwisted affine Lie algebras and
twisted affine Lie algebras, where
untwisted affine Lie algebras can be realized as the universal central extensions of  loop algebras of finite dimensional simple Lie algebras and twisted affine Lie algebras can be realized as fixed points subalgebras of untwisted affine Lie algebras under Dynkin digram automorphisms (see \cite{K}).
For extended affine Lie algebras,  essentially this is also the case;
it was proved (see \cite{ABFP, N1, N2, BGK, BGKN, Y}) that
almost all extended affine Lie algebras (except those constructed from the centerless irrational Lie tori) can be realized
as twisted toroidal Lie algebras. In view of this, almost all extended affine Lie algebras can be associated to toroidal vertex algebras.

Now, we give a more detailed account of the contents of this paper.
First, an $(r+1)$-toroidal vertex algebra (with $r$ a positive integer) is defined (see \cite{LTW})
to be a vector space $V$ equipped with a linear map
\begin{eqnarray*}
Y(\cdot;x_0,\x):&& V\rightarrow \Hom(V,V[[ \varr{x}]
]((x_0)))\\
&& v\mapsto Y(v;x_0,\x),
\end{eqnarray*}
where $\x=(x_1,\dots, x_r)$, and equipped with a distinguished vector ${\bf 1}$ of $V$, such that
\begin{eqnarray*}
\Y{\vac}{x_0}{\x}v=v&\andtext & \Y{v}{x_0}{\x}\vac\in V[[ x_0,\varr{x}]]\quad\fortext v\in V,
\end{eqnarray*}
and such that for $u,v\in V$,
\begin{eqnarray*}
&&\dfunc{z_0}{x_0-y_0}{z_0}Y(u;x_0,\z\y)Y(v;y_0,\y)-\dfunc{z_0}{y_0-x_0}{-z_0}Y(v;y_0,\y)Y(u;x_0,\z\y)\nonumber\\
&&\hspace{2cm}\quad= y_0^{-1}\delta\left(\frac{x_0-z_0}{y_0}\right) Y(Y(u;z_0,\z)v;y_0,\y).
\end{eqnarray*}
Now, let $V$ be an \rtva{} with a finite order automorphism $\sigma$ of period $N$.
We define a {\em $\sigma$-twisted $V$-module} to be a vector space $W$ equipped with a linear map
\begin{eqnarray*}
Y_W(\cdot;x_0,\x):&& V\rightarrow \Hom(W,W[[ \varr{x}]
]((x_0^{\frac{1}{N}})))\\
&& v\mapsto Y_W(v;x_0,\x)
\end{eqnarray*}
 such that
\begin{eqnarray*}
\Ymod{W}{\vac}{x_0}{\x}=1_W,
\end{eqnarray*}
and for $u,v\in V$,
\begin{eqnarray*}
&&\dfunc{z_0}{x_0-y_0}{z_0}\Ymod{W}{u}{x_0}{\z\y}\Ymod{W}{v}{y_0}{\y}-\dfunc{z_0}{y_0-x_0}{-z_0}\Ymod{W}{v}{y_0}{\y}\Ymod{W}{u}{x_0}{\z\y}\nonumber\\
&&\hspace{1cm}\quad=\frac{1}{N}\sum\limits_{j=0}^{N-1} y_0^{-1}\delta\left(  \rtu{N}^j \left(\frac{x_0-z_0}{y_0}\right)^{\frac{1}{N}}     \right) \Ymod{W}{\Y{\sigma ^j u}{z_0}{\z}v}{y_0}{\y}.
\end{eqnarray*}
We then give a conceptual construction of toroidal vertex algebras and twisted modules, analogous to a result of \cite{Li2}
for vertex algebras.
Let $W$ be a vector space and let $r$ and $N$ be positive integers. Set
\begin{eqnarray*}
\mathcal{E}(W,r;N)=\Hom(W,W[[\varr{x}]]((x_0^{\frac{1}{N}} )))\subset \End(W)[[x_0^{\pm
\frac{1}{N}},\varr{x}]].
\end{eqnarray*}
We equip the space $\mathcal{E}(W,r;N)$ with the following $\Z_{N}$-grading
\begin{eqnarray*}
\mathcal{E}(W,r;N)=\mathop{\bigoplus}\limits_{s=0}^{N-1}\mathcal{E}(W,r;N)_s,
\end{eqnarray*}
where
$\mathcal{E}(W,r;N)_s=x_0^{-\frac{s}{N}}\Hom(W,W[[\varr{x}]]((x_0)))$.
Let $\sigma$ be the linear automorphism of $\mathcal{E}(W,r;N)$ with $\mathcal{E}(W,r;N)_s$ as the eigenspace
of eigenvalue $e^{2\pi is/N}$ for $0\le s<N$. 
We call a subset $U$ of $\mathcal{E}(W,r;N)$ local if
for any $a(x_0,\x),b(x_0,\x)\in U$, there exists a nonnegative
integer $k$ such that
\begin{eqnarray*}
(x_0-y_0)^k \left[ a(x_0,\x),b(y_0,\y) \right]=0.
\end{eqnarray*}
As the main result, we prove that every graded local subspace $U$ of $\mathcal{E}(W,r;N)$ generates an \rtva{} $\<U\>$ in
a certain canonical way with $\sigma$ as an automorphism of period $N$ and that $W$ is a natural $\sigma$-twisted module
for $\<U\>$.

The second half of this paper is devoted to establishing
an association of toroidal vertex algebras to twisted toroidal Lie algebras.
Let $\g$ be a finite dimensional simple Lie algebra and let $\sigma_0,\sigma_1,\dots,\sigma_r$
be mutually commuting automorphisms of
$\g$ with finite orders $N_0,N_1,\dots,N_r$, respectively.
Set $G=\langle \sigma_0,\sigma_1,\dots,\sigma_r \rangle\subset \mbox{Aut}(\g)$.
Consider the following  $(r+1)$-loop Lie algebra
$$L_{r+1}(\g,N_0)=\g\otimes \C[t_0^{\pm \frac{1}{N_{0}} },\varr{t} ]$$
and its central extension
\begin{eqnarray*}
\widehat{L_{r+1}}(\g,N_{0})=\g\otimes \C[t_0^{\pm \frac{1}{N_{0}} },\varr{t}
]\oplus \C\cent.
\end{eqnarray*}
Let $\tau$ be the $G$-fixed points subalgebra
\begin{eqnarray*}
\tau=\left(\widehat{L_{r+1}}(\g,N_{0})\right)^{G}.
\end{eqnarray*}
Note that it was proved in \cite{ABFP} that for almost all extended affine Lie
algebras of nullity $r+1$,
the centerless core is isomorphic to such a Lie algebra $\tau$.

To associate \rtva{}s to twisted toroidal Lie algebra $\tau$, we consider a subalgebra of
the $(r+1)$-toroidal Lie algebra $\widehat{L_{r+1}}(\g)=\widehat{L_{r+1}}(\g,1)$.  Set
\begin{eqnarray*}
\mathcal{L}=\left(\mathop{\bigoplus}\limits_{\m\in \Zr}\g_\m\otimes  \bdt^\m\C[
t_0^{\pm 1}]\right) \bigoplus \C\cent,
\end{eqnarray*}
where $\g_{\m}=\set{a\in\g}{\sigma_j(a)=\rtu{N_j}^{m_j}a\  \  \mbox{ for  }1\leq j\leq r }$
for $\m=(m_{1},\dots,m_{r})\in \Zr$.
Set
\begin{eqnarray*}
\mathcal{L}^{\ge 0}=\left(\mathop{\bigoplus}\limits_{\m\in\Zr}\g_\m\otimes
\bdt^\m\C\left[t_0 \right] \right)\bigoplus \C\cent,
\end{eqnarray*}
a subalgebra of $\mathcal{L}$.
For any complex number $\ell$, we construct an  \rtva{} $V_\mathcal{L}(\ell,0)$ whose
underlying vector space is the following induced module
\begin{eqnarray*}
V_\mathcal{L}(\ell,0)=U\left( \mathcal{L} \right)\otimes_{U(\mathcal{L}^{\ge 0})}
\left(\g+\C \right),
\end{eqnarray*}
where $\g\oplus \C$ is equipped with a suitably defined $\mathcal{L}^{\ge 0}$-module structure such that $\cent$
acts as scalar $\ell$.
The automorphism $\sigma_0$ of $\g$ is shown to induce an
automorphism $\tilde{\sigma}$ of $V_{\mathcal{L}}(\ell,0)$ with order $N_0$.
%
Then we show that the category of restricted
$\tau$-modules of level $\ell$ is naturally isomorphic to that of $\tilde{\sigma}$-twisted
$V_\mathcal{L}(\ell,0)$-modules satisfying a certain equivariance property (see Theorem \ref{last-thm} for details).

This paper is organized as follows: In Section 2, we define the notation of
twisted module for a general toroidal vertex algebra and we present some basic
results. In Section 3, we give a general construction of toroidal vertex
algebras and their twisted modules. In Section 4, we associate toroidal
vertex algebras and their twisted modules to twisted toroidal Lie algebras.

\section{Twisted modules for toroidal vertex algebras}
In this section, we define the notion of  twisted module for a toroidal vertex algebra
with a finite order automorphism and we present some basic properties for twisted modules.

First of all, throughout this paper, we denote by $\C$, $\N$, and $\mathbb{Z}$
the field of complex numbers, the set of nonnegative integers, the set of integers, respectively.
The symbols $x,y,z,x_{0},y_{0},z_0,x_{1},y_{1},z_1,\dots$ denote mutually commuting independent formal variables.
All vector spaces in this paper are considered to be over  $\mathbb{C}$.

Let $r$ be a positive integer which is fixed throughout this paper. For any $\m=(m_1,\dots,m_r)\in \Zr$, we set
\begin{eqnarray*}
\x^\m=x_1^{m_1}\cdots x_r^{m_r}.
\end{eqnarray*}
As a convention, we write
\begin{eqnarray*}
\x^{-1}=x_1^{-1}\cdots x_r^{-1}, && \x^{\m-1}=x_1^{m_1-1}\cdots x_r^{m_r-1},
\end{eqnarray*}
and
\begin{eqnarray*}
\Res_\x=\Res_{x_1}\cdots \Res_{x_r}.
\end{eqnarray*}
For a vector space $W$ and a positive integer $N$, we set
\begin{eqnarray}
\mathcal{E}(W,r;N)=\Hom ( W, W[[\varr{x}]](( x_0^{\frac{1}{N}}))).
\end{eqnarray}
In particular, we write $\ESall{W}$ for $\ESalltwistedtemp{W}$, i.e.,
\begin{eqnarray}
\ESall{W}=\Hom (W, W[[\varr{x}]](( x_0))).
\end{eqnarray}

Recall the formal delta-function
$$\delta(z)=\sum_{n\in\mathbb{Z}}z^{n}\in \C[[z,z^{-1}]].$$
We have
 \begin{eqnarray*}
\delta\left(\frac{z_1-z_2}{z_{0}}\right)=\sum_{n\in\mathbb{Z}}z_{0}^{-n}(z_{1}-z_{2})^{n}
=\sum_{n\in\mathbb{Z}}\sum_{i\in\mathbb{N}}(-1)^{i}\choice{n}{i}z_{0}^{-n}z_{1}^{n-i}z_{2}^{i}.
\end{eqnarray*}
The following is one of the basic properties of the delta-function
\begin{eqnarray}\label{eq:delta.}
z_{0}^{-1}\delta\left(\frac{z_1-z_2}{z_{0}}\right)\left(\frac{z_1-z_2}{z_{0}}\right)^{\alpha}
=z_{1}^{-1}\delta\left(\frac{z_0+z_2}{z_{1}}\right)\left(\frac{z_0+z_2}{z_{1}}\right)^{-\alpha}
\end{eqnarray}
for $\alpha\in\C$ (see \cite{DL}).

Next, we recall the definition of an \rtva{} from \cite{LTW}.

\begin{de}
An {\em \rtva{}} is a vector space $V$ equipped with a linear map
\begin{eqnarray*}
\Y{\cdot}{x_0}{\x}: &&V\rightarrow \ESall{V},\\
                    &&v\mapsto \Y{v}{x_0}{\x}=\sum_{(m_0,\m)\in \Z\times \Z^{r}}v_{\suball{m_0}{\m}}x_0^{-m_0-1}\x^{-\m}
\end{eqnarray*}
and equipped with a distinguished vector $\vac\in V$, satisfying the conditions that
\begin{eqnarray*}
\Y{\vac}{x_0}{\x}v=v&\andtext & \Y{v}{x_0}{\x}\vac\in V[[ x_0,\varr{x}]]\quad\fortext v\in V
\end{eqnarray*}
and that for $u,v\in V$,
\begin{eqnarray}\label{ejacobi-rva}
&&\dfunc{z_0}{x_0-y_0}{z_0}\Y{u}{x_0}{\z\y}\Y{v}{y_0}{\y}-\dfunc{z_0}{y_0-x_0}{-z_0}\Y{v}{y_0}{\y}\Y{u}{x_0}{\z\y} \nonumber \\
&&\hspace{2cm}=\dfunc{y_0}{x_0-z_0}{y_0}\Y{\Y{u}{z_0}{\z}v}{y_0}{\y},
\end{eqnarray}
where
\begin{eqnarray*}
\Y{u}{x_0}{\z\y}=\sum_{(m_0,\m)\in \Z\times \Z^{r}}u_{\suball{m_0}{\m}}x_0^{-m_0-1}\z^{-\m}\y^{-\m}.
\end{eqnarray*}
\end{de}

Let $u\in V$. We define $Y(u;m_0,\x)$ and $Y(u;x_0,\m)$ for  $m_0\in \Z,\ \m\in \Z^r$ by
\begin{eqnarray}\label{edef-summand}
Y(u;x_0,\x)=\sum_{m_0\in \Z}Y(u;m_{0},\x)x_0^{-m_0-1}=\sum_{\m\in \Z^{r}}Y(u;x_{0},\mathbf{m})\x^{-\m}.
\end{eqnarray}
From the Jacobi identity (\ref{ejacobi-rva}) we get
\begin{eqnarray}\label{ecommutator-rva}
&&[Y(u;x_{0},\z\y), Y(v,y_{0},\y)]\nonumber\\
&=&\Res_{z_0}\sum_{j\ge 0}\frac{1}{j!}\left(\frac{\partial}{\partial y_{0}}\right)^{j}
\left(x_{0}^{-1}\delta\left(\frac{y_{0}}{x_{0}}\right)\right)z_0^{j}Y(Y(u;z_0,\z)v;y_{0},\y)\nonumber\\
&=&\sum_{j\ge 0}\frac{1}{j!}\left(\frac{\partial}{\partial y_{0}}\right)^{j}
\left(x_{0}^{-1}\delta\left(\frac{y_{0}}{x_{0}}\right)\right)Y(Y(u;j,\z)v;y_{0},\y).
\end{eqnarray}
Furthermore, we have
\begin{eqnarray}\label{ecommutator-rva-m}
\left[ Y(u;x_0,\m),Y(v;y_0,\y) \right]=\y^\m\sum\limits_{j\ge 0}
Y(u_{j,\m}v;y_0,\y)\frac{1}{j!}
\left(\frac{\partial}{\partial y_0} \right)^j \dfunc{x_0}{y_0}{x_0}
\end{eqnarray}
for $\m\in \Z^r$.

Let $V_1$ and $V_2$ be \rtva{}s. An \rtva{} {\em homomorphism} from $V_1$ to $V_2$ is a linear map $\sigma$ such that
\begin{eqnarray*}
\sigma(\mathbf{1})=\mathbf{1}\   \mbox{ and } \   \sigma(Y(u;x_0,\x)v)=Y(\sigma(u);x_0,\x)\sigma(v)
\   \   \   \mbox{ for }u,v\in V_1.
\end{eqnarray*}
An {\em automorphism} of an \rtva{} $V$ is defined to be a bijective homomorphism from $V$ to $V$.

Next, we define a notion of twisted module for an \rtva{} $V$.
Let $\sigma$ be a finite order automorphism of $V$ and let $N$ be a period of  $\sigma$.
(Here, $N$ is a positive integer such that $\sigma^{N}=1$, but $N$ is not necessarily the order of $\sigma$.) Set
$$\rtu{N}=\exp \left( 2\pi\sqrt{ -1}/N \right),$$
 the principal primitive $N$-th root of unity.
Then $V=\oplus_{s=0}^{N-1}V^s$, where $V^j=\set{u\in V}{\sigma (u)=\rtu{N}^{j} u}$ for any $j\in \Z$.

\begin{de}
A {\em $\sigma$-twisted $V$-module} is a vector space $W$ equipped with a linear map
\begin{eqnarray*}
\Ymod{W}{\cdot}{x_0}{\x}:&& V\rightarrow \mathcal{E}(W,r;N)\\
                        &&  v\mapsto \Ymod{W}{v}{x_0}{\x}=\sum_{(m_0,\m)\in\frac{1}{N}\mathbb{Z}\times{\mathbb{Z}}^{r}}v_{\suball{m_0}{\m}}x_0^{-m_0-1}\x^{-\m},
\end{eqnarray*}
such that
\begin{eqnarray*}
\Ymod{W}{\vac}{x_0}{\x}=1_W
\end{eqnarray*}
and such that the following $\sigma$-twisted Jacobi identity holds for $u,v\in V$:
\begin{eqnarray}\label{eq:twistedJacobian}
&&\dfunc{z_0}{x_0-y_0}{z_0}\Ymod{W}{u}{x_0}{\z\y}\Ymod{W}{v}{y_0}{\y}-\dfunc{z_0}{y_0-x_0}{-z_0}\Ymod{W}{v}{y_0}{\y}\Ymod{W}{u}{x_0}{\z\y}\nonumber\\
&&\  \  \  \
\quad=\frac{1}{N}\sum\limits_{j=0}^{N-1} y_0^{-1}\delta\left(  \rtu{N}^j \left(\frac{x_0-z_0}{y_0}\right)^{\frac{1}{N}}     \right) \Ymod{W}{\Y{\sigma ^j u}{z_0}{\z}v}{y_0}{\y}.
\end{eqnarray}
\end{de}

\begin{lem}\label{ehom-power}
Let $(W,Y_{W})$ be a $\sigma$-twisted $V$-module and let $u,v\in V$. Assume $u\in V^s$ with $0\leq s<N$.
Then
\begin{eqnarray}\label{eq:VertexOpSpaceHomo}
\Ymod{W}{u}{x_0}{\x}\in x_{0}^{-\frac{s}{N}}\ESall{W}.
\end{eqnarray}
\end{lem}

\begin{proof} As $u\in V^s$, from (\ref{eq:twistedJacobian}) we get
\begin{eqnarray}\label{eq:twistedJacobianHomo}
&&\dfunc{z_0}{x_0-y_0}{z_0}\Ymod{W}{u}{x_0}{\z\y}\Ymod{W}{v}{y_0}{\y}-\dfunc{z_0}{y_0-x_0}{-z_0}\Ymod{W}{v}{y_0}{\y}\Ymod{W}{u}{x_0}{\z\y}\nonumber\\
&&\   \   \  \   \quad= \dfunc{y_0}{x_0-z_0}{y_0}\left( \frac{x_0-z_0}{y_0} \right)^{-\frac{s}{N}} \Ymod{W}{\Y{u}{z_0}{\z}v}{y_0}{\y}.
\end{eqnarray}
Using (\ref{eq:delta.}),  one can also write (\ref{eq:twistedJacobianHomo})  as
\begin{eqnarray}\label{eq:twistedJacobianHomo1}
&&\dfunc{z_0}{x_0-y_0}{z_0}\Ymod{W}{u}{x_0}{\z\y}\Ymod{W}{v}{y_0}{\y}-\dfunc{z_0}{y_0-x_0}{-z_0}\Ymod{W}{v}{y_0}{\y}\Ymod{W}{u}{x_0}{\z\y}\nonumber\\
&&\  \  \   \   \quad= \dfunc{x_0}{y_0+z_0}{x_0}\left( \frac{y_0+z_0}{x_0} \right)^{\frac{s}{N}} \Ymod{W}{\Y{u}{z_0}{\z}v}{y_0}{\y}.
\end{eqnarray}
Taking $v=\vac$ in (\ref{eq:twistedJacobianHomo}), we get
\begin{eqnarray*}
&&\dfunc{y_0}{x_0-z_0}{y_0}\Ymod{W}{u}{x_0}{\z\y}\\
&=&\dfunc{y_0}{x_0-z_0}{y_0}\left( \frac{x_0-z_0}{y_0} \right)^{-\frac{s}{N}} \Ymod{W}{\Y{u}{z_0}{\z}\vac}{y_0}{\y}.
\end{eqnarray*}
Then applying $\mbox{Res}_{z_{0}}z_{0}^{-1}$ (or setting $z_0=0$) we have
\begin{eqnarray*}
y_{0}^{-1}\delta\left(\frac{x_{0}}{y_{0}}\right)\Ymod{W}{u}{x_0}{\z\y}=
y_{0}^{-1}\delta\left(\frac{x_{0}}{y_{0}}\right)\left(\frac{x_{0}}{y_{0}}\right)^{-\frac{s}{N}}
\sum_{\mathbf{n}\in\mathbb{Z}^{r}}Y_{W}(u_{-1,\mathbf{n}}\mathbf{1};y_{0},\y){\z}^{-\mathbf{n}}.
\end{eqnarray*}
From this we get (\ref{eq:VertexOpSpaceHomo}).
\end{proof}

Let $(W,Y_{W})$ be a $\sigma$-twisted $V$-module and let $u\in V^s,\ v\in V$ as in Lemma \ref{ehom-power}.
Applying $\mbox{Res}_{z_{0}}$ to  (\ref{eq:twistedJacobianHomo1}), we obtain the following twisted commutator formula
\begin{eqnarray}\label{eq:commutator1}
&&[Y_{W}(u;x_{0},\z\y), Y_{W}(v,y_{0},\y)]\nonumber\\
&=&\Res_{z_0}\sum_{j\ge 0}\frac{1}{j!}\left(\frac{\partial}{\partial y_{0}}\right)^{j}
\left(x_{0}^{-1}\delta\left(\frac{y_{0}}{x_{0}}\right)\left(\frac{y_{0}}{x_{0}}\right)^{\frac{s}{N}}\right)
z_0^{j}Y_{W}(Y(u;z_0,\z)v;y_{0},\y)\  \  \  \  \  \  \nonumber\\
&=&\sum_{j\ge 0}\frac{1}{j!}\left(\frac{\partial}{\partial y_{0}}\right)^{j}
\left(x_{0}^{-1}\delta\left(\frac{y_{0}}{x_{0}}\right)\left(\frac{y_{0}}{x_{0}}\right)^{\frac{s}{N}}\right)
Y_{W}(Y(u;j,\z)v;y_{0},\y).
\end{eqnarray}
Furthermore, we have
\begin{eqnarray}\label{eq:commutator2}
&&[Y_{W}(u;x_{0},\mathbf{m}), Y_{W}(v,y_{0},\y)]\nonumber\\
&=&{\y}^{\mathbf{m}}\sum_{j\ge 0}\frac{1}{j!}\left(\frac{\partial}{\partial y_{0}}\right)^{j}
\left(x_{0}^{-1}\delta\left(\frac{y_{0}}{x_{0}}\right)\left(\frac{y_{0}}{x_{0}}\right)^{\frac{s}{N}}\right)
Y_{W}(u_{j,\mathbf{m}}v;y_{0},\y)
\end{eqnarray}
for $\m\in \Z^r$, where
\begin{eqnarray*}
Y_{W}(u;x_0,\x)=\sum_{\m\in \Z^{r}}Y_{W}(u;x_{0},\mathbf{m})\x^{-\m}.
\end{eqnarray*}
Multiplying (\ref{eq:twistedJacobianHomo}) by $\left(\frac{x_{0}-z_{0}}{y_{0}}\right)^{\frac{s}{N}}$,
 then applying $\mbox{Res}_{x_{0}}$, we get  a twisted iterate formula
\begin{eqnarray}\label{eq:iterate}
Y_{W}(Y(u;z_{0},\z)v;y_{0},\y)=\mbox{Res}_{x_{0}}\left(\frac{x_{0}-z_{0}}{y_{0}}\right)^{\frac{s}{N}}\cdot X,
\end{eqnarray}
where \begin{eqnarray*}
X=\dfunc{z_0}{x_0-y_0}{z_0}\Ymod{W}{u}{x_0}{\z\y}\Ymod{W}{v}{y_0}{\y}-\dfunc{z_0}{y_0-x_0}{-z_0}\Ymod{W}{v}{y_0}{\y}\Ymod{W}{u}{x_0}{\z\y}.
\end{eqnarray*}
If we use (\ref{eq:twistedJacobianHomo1}),  similarly we get the following variation
\begin{eqnarray}\label{eq:iterate1}
Y_{W}(Y(u;z_{0},\z)v;y_{0},\y)=\mbox{Res}_{x_{0}}\left(\frac{y_{0}+z_{0}}{x_{0}}\right)^{-\frac{s}{N}}\cdot X,
\end{eqnarray}
where $X$ is given as above.

\begin{lem}\label{lJ-weak}
Let $(W,Y_{W})$ be a $\sigma$-twisted $V$-module. (a) For $u,v\in V$, there exists a nonnegative integer $k$ such that
\begin{eqnarray}\label{eq:WeakCom}
\left(x_0-y_0\right)^k \left[\Ymod{W}{u}{x_0}{\x},\Ymod{W}{v}{y_0}{\y}\right]=0.
\end{eqnarray}
(b)  For $u\in V^s$, $v\in V$, $w\in W$, there exists a nonnegative integer $\ell$ such that
\begin{eqnarray}\label{eq:WeakAsso}
&&\left(z_0+y_0\right)^{\ell+\frac{s}{N}}\Ymod{W}{u}{z_0+y_0}{\z\y}\Ymod{W}{v}{y_0}{\y}w\nonumber\\
&=&\left(y_0+z_0\right)^{\ell+\frac{s}{N}} \Ymod{W}{\Y{u}{z_0}{\z}v}{y_0}{\y}w.
\end{eqnarray}
\end{lem}

\begin{proof}
For part (a), let $k\in \N$ be such that $z_0^{k}Y(u;z_0,\z)v\in V[[z_0,\z,\z^{-1}]]$.
Then the assertion follows from (\ref{eq:commutator1}) immediately.

To prove part (b), let $k\in \N$ be such that (\ref{eq:WeakCom}) holds and let $l\in \N$ be such that
$$x_{0}^{l+\frac{s}{N}}Y_{W}(u;x_{0},\z\y)w\in W[[x_{0},\y,{\y}^{-1},\z,{\z}^{-1}]].$$
Then

\begin{eqnarray*}
x_{0}^{l+\frac{s}{N}}(x_{0}-y_{0})^{k}Y_{W}(u;x_{0},\z\y)Y_{W}(v;y_{0},\y)w
=x_{0}^{l+\frac{s}{N}}(x_{0}-y_{0})^{k}Y_{W}(v;y_{0},\y)Y_{W}(u;x_{0},\z\y)w
\end{eqnarray*}
and
\begin{eqnarray*}
x_{0}^{l+\frac{s}{N}}(x_{0}-y_{0})^{k}Y_{W}(v;y_{0},\y)Y_{W}(u;x_{0},\z\y)w
\in W[[x_{0},y_0^{\pm 1},\y,{\y}^{-1},\z,{\z}^{-1}]],
\end{eqnarray*}
which imply
\begin{eqnarray*}
x_{0}^{l+\frac{s}{N}}(x_{0}-y_{0})^{k}Y_{W}(u;x_{0},\z\y)Y_{W}(v;y_{0},\y)w
\in W[[x_{0},y_0^{\pm 1}, \y,{\y}^{-1},\z,{\z}^{-1}]].
\end{eqnarray*}
In view of this, we have
\begin{eqnarray*}
&&\left[(x_{0}-y_{0})^{k}x_{0}^{l+\frac{s}{N}}Y_{W}(u;x_{0},\z\y)Y_{W}(v;y_{0},\y)w\right]|_{x_{0}=y_0+z_0}\nonumber\\
&=&\left[(x_{0}-y_{0})^{k}x_{0}^{l+\frac{s}{N}}Y_{W}(u;x_{0},\z\y)Y_{W}(v;y_{0},\y)w\right]|_{x_{0}=z_0+y_0}.
\end{eqnarray*}
Using (\ref{eq:iterate1}), delta-function substitution, and (\ref{eq:WeakCom}), we obtain
\begin{eqnarray*}
&&z_{0}^{k}(y_{0}+z_{0})^{l+\frac{s}{N}}Y_{W}(Y(u;z_{0},\z)v;y_{0},\y)w\nonumber\\
&=&\Res_{x_{0}}x_{0}^{-1}\delta\left(\frac{y_{0}+z_{0}}{x_{0}}\right)\left(\frac{y_{0}+z_{0}}{x_{0}}\right)
^{-\frac{s}{N}}(y_{0}+z_{0})^{l+\frac{s}{N}}\\
&&\    \cdot \left[(x_{0}-y_{0})^{k}Y_{W}(u;x_{0},\z\y)Y_{W}(v;y_{0},\y)w\right]
\nonumber\\
&=&\Res_{x_{0}}x_{0}^{-1}\delta\left(\frac{y_{0}+z_{0}}{x_{0}}\right)
\left[(x_{0}-y_{0})^{k}x_{0}^{l+\frac{s}{N}}Y_{W}(u;x_{0},\z\y)Y_{W}(v;y_{0},\y)w\right]\nonumber\\
&=&\left[(x_{0}-y_{0})^{k}x_{0}^{l+\frac{s}{N}}Y_{W}(u;x_{0},\z\y)Y_{W}(v;y_{0},\y)w\right]|_{x_{0}=y_0+z_0}\nonumber\\
&=&\left[(x_{0}-y_{0})^{k}x_{0}^{l+\frac{s}{N}}Y_{W}(u;x_{0},\z\y)Y_{W}(v;y_{0},\y)w\right]|_{x_{0}=z_0+y_0}\nonumber\\
&=&z_{0}^{k}(z_{0}+y_{0})^{l+\frac{s}{N}}Y_{W}(u;z_{0}+y_{0},\z\y)Y_{W}(v;y_{0},\y)w.\nonumber
\end{eqnarray*}
Multiplying both sides by $z_{0}^{-k}$, we obtain (\ref{eq:WeakAsso}).
\end{proof}

The property (a) in Lemma \ref{lJ-weak} is called {\em weak commutativity,} while property (b)
 is called {\em weak twisted associativity.}
Just as with twisted modules for vertex algebras (see \cite{Li2}), the converse of Lemma \ref{lJ-weak} also holds.

\begin{lem}\label{lem:JacobianComAsso}
Let $V$ be an \rtva{} and let $\sigma$ be an automorphism of period  $N$. In the definition of
a $\sigma$-twisted $V$-module, the twisted Jacobi identity can be equivalently replaced
by the weak commutativity and the twisted weak associativity.
\end{lem}

\begin{proof}
Let $u\in V^s$, $v\in V$ and $w\in W$ with $0\le s<N$. Let $k\in \N$ be
such that (\ref{eq:WeakCom}) and (\ref{eq:WeakAsso}) hold and
such that $x_{0}^{k+\frac{s}{N}}Y_{W}(u;x_{0},\x)w\in W[[x_0,\x,\x^{-1}]]$. Then
\begin{eqnarray*}
(x_0-y_0)^{k}x_{0}^{k+\frac{s}{N}}\Ymod{W}{u}{x_0}{\z\y}\Ymod{W}{v}{y_0}{\y}w
=(x_0-y_0)^{k}x_{0}^{k+\frac{s}{N}}\Ymod{W}{v}{y_0}{\y}\Ymod{W}{u}{x_0}{\z\y}w.
\end{eqnarray*}
As the expression on the right-hand side involves only nonnegative integer powers of $x_0$, so does the expression on the left-hand side. Thus
\begin{eqnarray*}
&&\left[(x_0-y_0)^{k}x_{0}^{k+\frac{s}{N}}\Ymod{W}{u}{x_0}{\z\y}\Ymod{W}{v}{y_0}{\y}w\right]|_{x_0=y_0+z_0}\nonumber\\
&=&\left[(x_0-y_{0})^{k}x_{0}^{k+\frac{s}{N}}
\Ymod{W}{u}{x_0}{\z\y}\Ymod{W}{v}{y_0}{\y}w\right]|_{x_0=z_0+y_0}.
\end{eqnarray*}
Then using delta-function substitution and (\ref{eq:WeakAsso}) we get
\begin{eqnarray}\label{eqtemp:1:lem}
&&z_0^k x_{0}^{k}\dfunc{z_0}{x_0-y_0}{z_0}\Ymod{W}{u}{x_0}{\z\y}\Ymod{W}{v}{y_0}{\y}w\nonumber\\
&&\quad \     -z_0^k x_{0}^{k}\dfunc{z_0}{y_0-x_0}{-z_0}\Ymod{W}{v}{y_0}{\y}\Ymod{W}{u}{x_0}{\z\y}w\nonumber\\
&=&\dfunc{z_0}{x_0-y_0}{z_0}x_{0}^{-\frac{s}{N}}\left[(x_0-y_0)^{k}x_{0}^{k+\frac{s}{N}}\Ymod{W}{u}{x_0}{\z\y}\Ymod{W}{v}{y_0}{\y}w\right]\nonumber\\
&&\quad \   -\dfunc{z_0}{y_0-x_0}{-z_0}x_{0}^{-\frac{s}{N}}
\left[(x_0-y_0)^{k}x_{0}^{k+\frac{s}{N}}\Ymod{W}{v}{y_0}{\y}\Ymod{W}{u}{x_0}{\z\y}w\right]\nonumber\\
&=&\dfunc{y_0}{x_0-z_0}{y_0}x_{0}^{-\frac{s}{N}}
\left[(x_0-y_0)^{k}x_{0}^{k+\frac{s}{N}}\Ymod{W}{u}{x_0}{\z\y}\Ymod{W}{v}{y_0}{\y}w\right]\nonumber\\
&=&\dfunc{x_0}{y_0+z_0}{x_0}x_{0}^{-\frac{s}{N}}\left[(x_0-y_{0})^{k}x_{0}^{k+\frac{s}{N}}
\Ymod{W}{u}{x_0}{\z\y}\Ymod{W}{v}{y_0}{\y}w\right]|_{x_0=y_0+z_0}\nonumber\\
&=&\dfunc{x_0}{y_0+z_0}{x_0}x_{0}^{-\frac{s}{N}}\left[(x_0-y_{0})^{k}x_{0}^{k+\frac{s}{N}}
\Ymod{W}{u}{x_0}{\z\y}\Ymod{W}{v}{y_0}{\y}w\right]|_{x_0=z_0+y_0}\nonumber\\
&=&\dfunc{x_0}{y_0+z_0}{x_0}
x_{0}^{-\frac{s}{N}}z_{0}^{k}(z_{0}+y_{0})^{k+\frac{s}{N}}\Ymod{W}{u}{z_0+y_0}{\z\y}
\Ymod{W}{v}{y_0}{\y}w\nonumber\\
&=&\dfunc{x_0}{y_0+z_0}{x_0}x_{0}^{-\frac{s}{N}}z_{0}^{k}(y_{0}+z_{0})^{k+\frac{s}{N}}Y_{W}(Y(u;z_{0},z)v;y_{0},\y)w
\nonumber\\
&=&x_{0}^{k}z_{0}^{k}\dfunc{x_0}{y_0+z_0}{x_0}\left(\frac{y_0+z_0}{x_0}\right)^{\frac{s}{N}}Y_{W}(Y(u;z_{0},z)v;y_{0},\y)w\nonumber\\
&=&x_{0}^{k}z_{0}^{k}\dfunc{y_0}{x_0-z_0}{y_0}\left(\frac{x_0-z_0}{y_0}\right)^{-\frac{s}{N}}Y_{W}(Y(u;z_{0},z)v;y_{0},\y))w.\nonumber
\end{eqnarray}
Multiplying by $x_{0}^{-k}z_{0}^{-k}$, we obtain the twisted Jacobi identity.
\end{proof}

Using (\ref{eq:commutator2}), Lemma 2.3 in \cite{Li2}, and (\ref{ecommutator-rva-m}), we immediately  have
the following analogue of  one half of Lemma 2.11 therein:

\begin{lem}\label{ltwistedm-rva}
Let $V$ be an  \rtva{} with a finite order automorphism $\sigma$ of period $N$
and let $(W,Y_{W})$ be a faithful $\sigma$-twisted $V$-module. Let
\begin{eqnarray*}
a,b,c^{(0)},c^{(1)},\dots,c^{(k)}\in V,\   \m\in \Zr
\end{eqnarray*}
and assume $a\in V^s$ with $0\le s<N$. If
\begin{eqnarray}\label{ebracket-onmodule}
&&\left[ Y_{W}(a;x_0,\m),Y_{W}(b;y_0,\y)\right]\nonumber\\
&=&\y^\m \sum\limits_{j=0}^k Y_{W}(c^{(j)};y_0,\y)\frac{1}{j!}\left(\frac{\partial}{\partial y_0} \right)^j
\left(\dfunc{x_0}{y_0}{x_0}\left(\frac{y_0}{x_0}\right)^{\frac{s}{N}} \right)
\end{eqnarray}
on $W$, then we have
$$a_{j,\m}b=c^{(j)}\   \   \mbox{ for }0\le j\le k\  \mbox{ and }\  a_{j,\m}b=0\   \   \mbox{ for }j>k,$$
and
\begin{eqnarray}
\left[ Y(a;x_0,\m),Y(b;y_0,\y) \right]=\y^\m\sum\limits_{j=0}^k
Y(c^{(j)};y_0,\y)\frac{1}{j!}
\left(\frac{\partial}{\partial y_0} \right)^j \dfunc{x_0}{y_0}{x_0}
\end{eqnarray}
on $V$.
\end{lem}

\begin{rem}
{\em Note that for twisted modules for a vertex algebra, the converse of Lemma \ref{ltwistedm-rva} is also true (see \cite{Li2})
as a vertex algebra itself is always a faithful module. However, this is not the case here;
an \rtva{} itself is not necessarily a faithful module.
Due to this,  we cannot claim the converse of Lemma \ref{ltwistedm-rva}. }
\end{rem}

\section{Construction of toroidal vertex algebras and their twisted modules}\label{Section:ConceptionalConstruction}
In this section, we give a general construction of \rtva{}s and their twisted modules
from local subsets of $\mathcal{E}(W,r;N)$ with $W$ being an arbitrary vector space.

Let $W$ be a vector space in addition to positive integers $r$ and $N$, which are all fixed throughout this section.
Recall
\begin{eqnarray*}
\mathcal{E}(W,r;N)=\Hom ( W, W[[\varr{x}]](( x_0^{\frac{1}{N}}))).
\end{eqnarray*}
We equip $\End(W)[[x_0^{\pm \frac{1}{N}},\varr{x}]]$ with the following $\Z_{N}$-grading
\begin{eqnarray}
\End(W)[[x_0^{\pm \frac{1}{N}},\varr{x}]]=\oplus_{[j]\in \Z_{N}}x^{-\frac{j}{N}}\End(W)[[x_0^{\pm 1},\varr{x}]].
\end{eqnarray}
Let $\sigma$ be the corresponding linear automorphism of $\End(W)[[x_0^{\pm
\frac{1}{N}},\varr{x}]]$, with $x^{-\frac{j}{N}}\End(W)[[x_0^{\pm 1},\varr{x}]]$
as the eigenspace of eigenvalue $\omega_{N}^{j}$
for $0\le j<N$. Namely,
\begin{eqnarray}
\sigma(f(x_{0}^{\frac{1}{N}},x_{1},\dots,x_{r}))=f(\omega_{N}^{-1}x_{0}^{\frac{1}{N}},x_{1},\dots,x_{r}).
\end{eqnarray}
We see that  a subspace of $\End(W)[[x_0^{\pm \frac{1}{N}},\varr{x}]]$ is $\sigma$-stable if and only if it is graded.
It is clear that $\mathcal{E}(W,r;N)$ is a graded subspace, so that $\sigma$ is a linear automorphism.
We have
$$\mathcal{E}(W,r;N)=\oplus_{0\le j\le N-1}\mathcal{E}(W,r;N)_{j},$$
where
\begin{eqnarray}
\mathcal{E}(W,r;N)_{j}=x_0^{-\frac{j}{N}}\Hom ( W, W[[\varr{x}]](( x_0)))=x_0^{-\frac{j}{N}}\mathcal{E}(W,r).
\end{eqnarray}

Let $a(x_0,\x),b(x_0,\x)\in  \mathcal{E}(W,r;N)$. We say that $a(x_0,\x)$ and $b(x_0,\x)$ are {\em mutually local}
 if there exists a nonnegative integer $k$ such that
\begin{eqnarray}\label{eab-locality}
(x_0-y_0)^k \left[a(x_0,\x),b(y_0,\y)\right]=0.
\end{eqnarray}
Furthermore, we say a subset $U$ of $\mathcal{E}(W,r;N)$ is {\em local}
if for any $a(x_0,\x),b(x_0,\x)\in U$, $a(x_0,\x)$ and $b(x_0,\x)$ are mutually local.

\begin{de}\label{def:defofYe}
Let $a(x_0,\x),b(x_0,\x)\in  \mathcal{E}(W,r;N)$. Assume that $a(x_0,\x)$ and $b(x_0,\x)$ are mutually local and
assume $a(x_0,\x)\in  \mathcal{E}(W,r;N)_j$
with $0\le j<N$. Then we define
$$a(y_0,\y)_{\suball{m_0}{\m}}b(y_0,\y)\in \mathcal{E}(W,r;N)\  \  \mbox{ for }(m_0,\m)\in \Z\times \Z^r$$
in terms of generating function
\begin{eqnarray*}
\Ye{a(y_0,\y)}{z_0}{\z}b(y_0,\y)=\sum_{(m_0,\m)\in \Z\times \Z^r}a(y_0,\y)_{\suball{m_0}{\m}}b(y_0,\y)z_0^{-m_0-1 }\z^{-\m}
\end{eqnarray*}
by
\begin{eqnarray}\label{eq:defofYe}
&&\Ye{a(y_0,\y)}{z_0}{\z}b(y_0,\y)=\Res_{x_0}\left(\frac{x_0-z_0}{y_0} \right)
^{\frac{j}{N}}\cdot X,
\end{eqnarray}
where
\begin{eqnarray*}
&&X=\dfunc{z_0}{x_0-y_0}{z_0} a(x_0,\z\y)b(y_0,\y)
-\dfunc{z_0}{y_0-x_0}{-z_0}b(y_0,\y)a(x_0,\z\y).
\end{eqnarray*}
\end{de}

\begin{lem}\label{rem:Operators}
Assume that $a(x_0,\x)\in \mathcal{E}(W,r;N)_{j},\ b(x_0,\x)\in \mathcal{E}(W,r;N)_{s}$ with $j,s\in \Z$ and that
$a(x_0,\x)$ and $b(x_0,\x)$ are mutually local.
Then
\begin{eqnarray}\label{eq:grading}
a(y_0,\y)_\suball{m_0}{\m}b(y_0,\y)\in \mathcal{E}(W,r;N)_{j+s}
\end{eqnarray}
for all $(m_0,\m)\in \Z\times \Z^r$ and $a(y_0,\y)_\suball{m_0}{\m}b(y_0,\y)=0$ whenever $m_0\ge k$,
where $k$ is a nonnegative integer such that (\ref{eab-locality}) holds.
\end{lem}

\begin{proof} By definition we have
\begin{eqnarray*}
y_0^{\frac{j+s}{N}}\Ye{a(y_0,\y)}{z_0}{\z}b(y_0,\y)=\Res_{x_0}(x_0-z_0)^{\frac{j}{N}}
\cdot y_0^{\frac{s}{N}}X,
\end{eqnarray*}
where $X$ is given as above. Since $y_0^{\frac{s}{N}}X$ involves only integer powers of $y_0$
(from Lemma \ref{ehom-power}), so does
$y_0^{\frac{j+s}{N}}\Ye{a(y_0,\y)}{z_0}{\z}b(y_0,\y)$. This proves the first assertion.
For the second assertion, note that for $m_0\in \Z$,
\begin{eqnarray*}
z_0^{m_0}\Ye{a(y_0,\y)}{z_0}{\z}b(y_0,\y)=\Res_{x_0}\sum_{i\ge 0}\choice{\frac{j}{N}}{i}(-1)^{i}x_{0}^{\frac{j}{N}-i}
y_0^{-\frac{j}{N}}\cdot \left(z_0^{i+m_0}X\right).
 \end{eqnarray*}
 Let $k\in \N$ be such that (\ref{eab-locality}) holds.
Noticing that for any integer $q$ with $q\ge k$,
$$\Res_{z_0}z_{0}^{q}X=(x_0-y_0)^{q}a(x_0,\z\y)b(y_0,\y)-(x_0-y_0)^{q}b(y_0,\y)a(x_0,\z\y)=0,$$
we have
$$\Res_{z_0}z_{0}^{m_0}\Ye{a(y_0,\y)}{z_0}{\z}b(y_0,\y)=0$$ for any $m_0\ge k$.
Thus, $a(x_0,\x)_{\suball{m_0}{\m}}b(x_0,\x)=0$ for any $(m_0,\m)\in\Z\times \Zr$ with $m_0\ge k$.
\end{proof}

\begin{rem}\label{rdifferent-def}
{\em  Let $a(x_0,\x)\in \mathcal{E}(W,r;N)_{j},\ b(x_0,\x)\in \mathcal{E}(W,r;N)_{s}$.
Suppose that $a(x_0,\x)$ and $b(x_0,\x)$ are mutually local with $k\in \N$ such that (\ref{eab-locality}) holds.
Then
\begin{eqnarray*}
(x_0-y_0)^{k}x_0^{\frac{j}{N}}a(x_0,\z\y)b(y_0,\y)=(x_0-y_0)^{k}x_0^{\frac{j}{N}}b(y_0,\y)a(x_0,\z\y),
\end{eqnarray*}
which (by using the information from two sides) implies
$$(x_0-y_0)^{k}x_0^{\frac{j}{N}}a(x_0,\z\y)b(y_0,\y)\in \Hom (W,W[[\y,\y^{-1},\z,\z^{-1}]]((x_0,y_0))).$$
Using delta-function substitution and (\ref{eab-locality}), we get
\begin{eqnarray}
z_0^{k}X=(x_0-y_0)^{k}X=\dfunc{y_0}{x_0-z_0}{y_0}\left[ (x_0-y_0)^{k}a(x_0,\z\y)b(y_0,\y)\right].
\end{eqnarray}
Then
\begin{eqnarray}
&&z_0^{k}\Ye{a(y_0,\y)}{z_0}{\z}b(y_0,\y)\nonumber\\
&=&\Res_{x_0}\left(\frac{x_0-z_0}{y_0} \right)^{\frac{j}{N}}\cdot z_0^{k}X\nonumber\\
&=&\Res_{x_0}\dfunc{y_0}{x_0-z_0}{y_0}\left(\frac{x_0-z_0}{y_0} \right)
^{\frac{j}{N}}\left[ (x_0-y_0)^{k}a(x_0,\z\y)b(y_0,\y)\right]\nonumber\\
&=&\Res_{x_0}\dfunc{x_0}{y_0+z_0}{x_0}\left(\frac{y_0+z_0}{x_0} \right)
^{-\frac{j}{N}}\left[ (x_0-y_0)^{k}a(x_0,\z\y)b(y_0,\y)\right]\nonumber\\
&=&(y_0+z_0)^{-\frac{j}{N}}\left[ (x_0-y_0)^{k}x_0^{\frac{j}{N}}a(x_0,\z\y)b(y_0,\y)\right]|_{x_0=y_0+z_0}.
\end{eqnarray}
Therefore, we obtain (cf. \cite{LTWqva}, Lemma 2.9)
\begin{eqnarray}\label{eY-different}
z_0^{k}(y_0+z_0)^{\frac{j}{N}}\Ye{a(y_0,\y)}{z_0}{\z}b(y_0,\y)=\left[(x_0-y_0)^{k}x_0^{\frac{j}{N}}a(x_0,\z\y)b(y_0,\y)\right]|_{x_0=y_0+z_0}.
\end{eqnarray}
This gives a different definition of $\Ye{a(y_0,\y)}{z_0}{\z}b(y_0,\y)$.}
\end{rem}

For any graded local subspace $U$ of $ \mathcal{E}(W,r;N)$, we
extend the definition linearly to define $a(y_0,\y)_{\suball{m_0}{\m}}b(y_0,\y)$ for any $a(x_0,\x),b(x_0,\x)\in U$.

The following is a key result to the construction:

\begin{lem}\label{lem:localclosed}
Assume that $a(x_0,\x),b(x_0,\x),c(x_0,\x)\in \mathcal{E}(W,r;N)$ are pairwise local.
Then for any $(m_0,\m)\in \Z\times \Zr$, $a(x_0,\x)_{\suball{m_0}{\m}}b(x_0,\x)$ and $c(x_0,\x)$ are local.
\end{lem}

\begin{proof} Fix $(m_0,\m)\in \Z\times \Zr$.
Let $k\in \N$ be such that $k+m_{0}>0$ and
\begin{eqnarray*}
&&(x_0-y_0)^k a(x_0,\x)b(y_0,\y)=(x_0-y_0)^k b(y_0,\y)a(x_0,\x),\\
&&(x_0-y_0)^k a(x_0,\x)c(y_0,\y)=(x_0-y_0)^k c(y_0,\y)a(x_0,\x),\\
&&(x_0-y_0)^k b(x_0,\x)c(y_0,\y)=(x_0-y_0)^k c(y_0,\y)b(x_0,\x).
\end{eqnarray*}
Assume $a(x_0,\x)\in \mathcal{E}(W,r;N)_s$ with $0\le s<N$. We have
\begin{eqnarray}
&&a(x_0,\x)_\suball{m_0}{\m}b(x_0,\x)\nonumber\\
&=&\Res_{y_0}{\x}^{-\m}\sum_{i=0}^{\infty}\choice{\frac{s}{N}}{i}\left(\frac{y_{0}}{x_{0}}\right)^{\frac{s}{N}}y_{0}^{-i}\cdot(-1)^{i}\cdot T_{i}\nonumber\\
&=&\Res_{y_0}{\x}^{-\m}\sum_{i=0}^{2k}\choice{\frac{s}{N}}{i}\left(\frac{y_{0}}{x_{0}}\right)^{\frac{s}{N}}y_{0}^{-i}\cdot(-1)^{i}\cdot T_{i},\nonumber
\end{eqnarray}
where
\begin{eqnarray*}
&&T_{i}=\left(y_0-x_0\right)^{m_0+i}
a(y_0,\m)b(x_0,\x)-\left(-x_0+y_0\right)^{m_0+i}
b(x_0,\x)a(y_0,\m).
\end{eqnarray*}
Since
\begin{eqnarray*}
&&(x_0-z_0)^{4k}((y_0-x_0)^{m_0+i}
a(y_0,\m)b(x_0,\x)c(z_0,\z)-(-x_0+y_0)^{m_0+i}
b(x_0,\x)a(y_0,\m)c(z_0,\z))\nonumber\\
&=&\sum_{j=0}^{3k}\choice{3k}{j}(x_{0}-y_{0})^{3k-j}(y_{0}-z_{0})^{j}(x_{0}-z_{0})^{k}\cdot
\nonumber\\
&&\quad\cdot((y_0-x_0)^{m_0+i}
a(y_0,\m)b(x_0,\x)c(z_0,\z)-(-x_0+y_0)^{m_0+i}
b(x_0,\x)a(y_0,\m)c(z_0,\z))\nonumber\\
&=&\sum_{j=k+1}^{3k}\choice{3k}{j}(x_{0}-y_{0})^{3k-j}(y_{0}-z_{0})^{j}(x_{0}-z_{0})^{k}\cdot
\nonumber\\
&&\quad\cdot((y_0-x_0)^{m_0+i}
a(y_0,\m)b(x_0,\x)c(z_0,\z)-(-x_0+y_0)^{m_0+i}
b(x_0,\x)a(y_0,\m)c(z_0,\z))\nonumber\\
&=&\sum_{j=k+1}^{3k}\choice{3k}{j}(x_{0}-y_{0})^{3k-j}(y_{0}-z_{0})^{j}(x_{0}-z_{0})^{k}\cdot
\nonumber\\
&&\quad\cdot((y_0-x_0)^{m_0+i}
c(z_0,\z)a(y_0,\m)b(x_0,\x)-(-x_0+y_0)^{m_0+i}
c(z_0,\z)b(x_0,\x)a(y_0,\m))\nonumber\\
&=&(x_0-z_0)^{4k}((y_0-x_0)^{m_0+i}
c(z_0,\z)a(y_0,\m)b(x_0,\x)\nonumber\\
&&\quad\quad-(-x_0+y_0)^{m_0+i}
c(z_0,\z)b(x_0,\x)a(y_0,\m)),\nonumber
\end{eqnarray*}
we get
\begin{eqnarray*}
(x_0-z_0)^{4k}(a(x_0,\x)_\suball{m_0}{\m}b(x_0,\x))c(z_0,\z)=(x_0-z_0)^{4k}c(z_0,\z)(a(x_0,\x)_\suball{m_0}{\m}b(x_0,\x)).
\end{eqnarray*}
as desired.
\end{proof}

A graded local subspace $U$ of $\mathcal{E}(W,r;N)$ is said to be
{\em closed} if
\begin{eqnarray*}
1_W\in U\    \mbox{ and }\  a(x_0,\x)_{\suball{m_0}{\m}}b(x_0,\x)\in U
\end{eqnarray*}
for all $a(x_0,\x),b(x_0,\x)\in U$, $(m_0,\m)\in \Z\times\Zr$.

The following is the main result of this section:

\begin{thm}\label{thm:ConceptionalConstruction}
Let $V$ be a closed graded local subspace of $\mathcal{E}(W,r;N)$.
Then $(V,Y_{\mathcal{E}},1_W)$ carries the structure of an \rtva{} with $\sigma$ as an automorphism
which has $N$ as a period and $W$
is a faithful $\sigma$-twisted $V$-module with
$$\Ymod{W}{a(x_0,\x)}{z_0}{\z}=a(z_0,\z)\   \   \mbox{ for }a(x_0,\x)\in V.$$
\end{thm}

\begin{proof} For any $b(x_0,\x)\in \mathcal{E}(W,r;N)$,  from Definition \ref{def:defofYe} we have
\begin{eqnarray*}
&&\Ye{1_W}{z_0}{\z}b(y_0,\y)=\Res_{x_0}\left(\dfunc{z_0}{x_0-y_0}{z_0}-\dfunc{z_0}{y_0-x_0}{-z_0}\right)b(y_0,\y)\\
&&\quad=\Res_{x_0}\dfunc{x_0}{y_0+z_0}{x_0}b(y_0,\y)=b(y_0,\y).
\end{eqnarray*}
On the other hand, for any $a(x_0,\x)\in\mathcal{E}(W,r;N)_{s}$ with $0\le s<N$, we have
\begin{eqnarray*}
&&\Ye{a(y_0,\y)}{z_0}{\z}1_W\\
&=&
 \Res_{x_0}\left( \frac{x_0-z_0}{y_0} \right)^{\frac{s}{N}}\left( \dfunc{z_0}{x_0-y_0}{z_0}-\dfunc{z_0}{y_0-x_0}{-z_0} \right)a(x_0,\z\y)\\
&=&\Res_{x_0}\dfunc{y_0}{x_0-z_0}{y_0}\left( \frac{x_0-z_0}{y_0} \right)^{\frac{s}{N}}a(x_0,\z\y)\\
&=&\Res_{x_0}\dfunc{x_0}{y_0+z_0}{x_0}\left( \frac{y_0+z_0}{x_0} \right)^{-\frac{s}{N}}a(x_0,\z\y)\\
&=&(y_0+z_0)^{-\frac{s}{N}}\Res_{x_0}\dfunc{x_0}{y_0+z_0}{x_0}\left(x_0^{\frac{s}{N}}a(x_0,\z\y\right)\\
&=&a(y_0+z_0,\z\y)=e^{z_0\frac{\partial}{\partial y_0}}a(y_0,\z\y).
\end{eqnarray*}

To prove the Jacobi identity, let $a(x_0,\x),b(x_0,\x),c(x_0,\x)\in V$ such that $a(x_0,\x)\in
\ESalltwisted{W}_\lambda$ and $b(x_0,\x)\in \ESalltwisted{W}_\mu$ with
$0\leq \lambda,\mu<N$.
For simplicity, in the following, we write $\Res_{u_0,v_0}$ for $\Res_{u_0}\Res_{v_0}$.
By Definition \ref{def:defofYe}, we have
\begin{eqnarray*}
&&\dfunc{z_0}{x_0-y_0}{z_0}\Ye{a(t_0,\bdt)}{x_0}{\y\z}\Ye{b(t_0,\bdt)}{y_0}{\y}c(t_0,\bdt)\\
&=&\Res_{u_0,v_0}\left( \frac{v_0-x_0}{t_0} \right)^{\frac{\lambda}{N}}\left(\frac{u_0-y_0}{t_0}\right)^{\frac{\mu}{N}}
    \dfunc{z_0}{x_0-y_0}{z_0} T,
\end{eqnarray*}
where
\begin{eqnarray*}
T&=&\dfunc{x_0}{v_0-t_0}{x_0}\dfunc{y_0}{u_0-t_0}{y_0}
a(v_{0},\z\y\bdt)b(u_{0},\y\bdt)c(t_{0},\bdt)\\
&&-\dfunc{x_0}{v_0-t_0}{x_0}\dfunc{y_0}{t_0-u_0}{-y_0}
a(v_{0},\z\y\bdt)c(t_{0},\bdt)b(u_{0},\y\bdt) \\
&&-\dfunc{x_0}{t_0-v_0}{-x_0}\dfunc{y_0}{u_0-t_0}{y_0}
b(u_{0},\y\bdt)c(t_{0},\bdt)a(v_{0},\z\y\bdt)\\
&&+\dfunc{x_0}{t_0-v_0}{-x_0}\dfunc{y_0}{t_0-u_0}{-y_0}
c(t_{0},\bdt)b(u_{0},\y\bdt)a(v_{0},\z\y\bdt).
\end{eqnarray*}
Let $k$ be a nonnegative integer such that
\begin{eqnarray*}
&&(x_0-y_0)^k\left[ a(x_0,\x),b(y_0,\y) \right]=0,\\
&&(x_0-y_0)^k\left[ a(x_0,\x),c(y_0,\y) \right]=0,\\
&&(x_0-y_0)^k\left[ b(x_0,\x),c(y_0,\y) \right]=0.
\end{eqnarray*}
Using the basic delta-function substitution property we get
\begin{eqnarray*}
&&x_0^ky_0^kz_0^k
    \dfunc{z_0}{x_0-y_0}{z_0} T=\dfunc{z_0}{x_0-y_0}{z_0}
    \left( v_0-t_0 \right)^k\left(u_0-t_0 \right)^k \left(v_0-u_0 \right)^k T   \\
&=&\dfunc{z_0}{x_0-y_0}{z_0}\dfunc{t_0}{v_0-x_0}{t_0}\dfunc{t_0}{u_0-y_0}{t_0}\\
&&\quad\cdot \left( v_0-t_0 \right)^k\left(u_0-t_0 \right)^k \left(v_0-u_0
\right)^k a(v_{0},\z\y\bdt)b(u_{0},\y\bdt)c(t_{0},\bdt).
\end{eqnarray*}
Consequently,
\begin{eqnarray*}
&&x_0^ky_0^kz_0^k\dfunc{z_0}{x_0-y_0}{z_0}\Ye{a(t_0,\bdt)}{x_0}{\y\z}\Ye{b(t_0,\bdt)}{y_0}{\y}c(t_0,\bdt)\\
&=&\Res_{u_0,v_0}\dfunc{z_0}{x_0-y_0}{z_0}\dfunc{t_0}{v_0-x_0}{t_0}\left( \frac{v_0-x_0}{t_0} \right)^{\frac{\lambda}{N}}
    \dfunc{t_0}{u_0-y_0}{t_0}\left(\frac{u_0-y_0}{t_0}\right)^{\frac{\mu}{N}}S,
\end{eqnarray*}
where
\begin{eqnarray*}
&&S=\left( v_0-t_0 \right)^k\left(u_0-t_0 \right)^k \left(v_0-u_0\right)^k
    a(v_{0},\z\y\bdt)b(u_{0},\y\bdt)c(t_{0},\bdt).
\end{eqnarray*}
Similarly, we have
\begin{eqnarray*}
&&x_0^ky_0^kz_0^k\dfunc{z_0}{y_0-x_0}{-z_0}\Ye{b(t_0,\bdt)}{y_0}{\y}\Ye{a(t_0,\bdt)}{x_0}{\y\z}c(t_0,\bdt)\\
&=&\Res_{u_0,v_0}\dfunc{z_0}{y_0-x_0}{-z_0}\dfunc{t_0}{v_0-x_0}{t_0}\left( \frac{v_0-x_0}{t_0} \right)^{\frac{\lambda}{N}}
    \dfunc{t_0}{u_0-y_0}{t_0}\left(\frac{u_0-y_0}{t_0}\right)^{\frac{\mu}{N}}S.
\end{eqnarray*}

On the other hand, we have
\begin{eqnarray*}
&&\dfunc{y_0}{x_0-z_0}{y_0}\Ye{\Ye{a(t_0,\bdt)}{z_0}{\z}b(t_0,\bdt)}{y_0}{\y}c(t_0,\bdt)\\
&=&\Res_{u_0,v_0}\left(\frac{v_0-z_0}{u_0}\right)^{\frac{\lambda}{N}}\left( \frac{u_0-y_0}{t_0} \right)^{\frac{\lambda+\mu}{N}}\dfunc{y_0}{x_0-z_0}{y_0}T',
\end{eqnarray*}
where
\begin{eqnarray*}
T'&=&\dfunc{z_0}{v_0-u_0}{z_0}\dfunc{y_0}{u_0-t_0}{y_0}a(v_0,\z\y\bdt)b(u_0,\y\bdt)c(t_0,\bdt)\\
&&-\dfunc{z_0}{u_0-v_0}{-z_0}\dfunc{y_0}{u_0-t_0}{y_0}b(u_0,\y\bdt)a(v_0,\z\y\bdt)c(t_0,\bdt)\\
&&-\dfunc{z_0}{v_0-u_0}{z_0}\dfunc{y_0}{t_0-u_0}{y_0}c(t_0,\bdt)a(v_0,\z\y\bdt)b(u_0,\y\bdt)\\
&&+\dfunc{z_0}{u_0-v_0}{-z_0}\dfunc{y_0}{t_0-u_0}{y_0}c(t_0,\bdt)b(u_0,\y\bdt)a(v_0,\z\y\bdt),
\end{eqnarray*}
and furthermore we have
\begin{eqnarray*}
&&x_0^ky_0^kz_0^k \dfunc{y_0}{x_0-z_0}{y_0}\Ye{\Ye{a(t_0,\bdt)}{z_0}{\z}b(t_0,\bdt)}{y_0}{\y}c(t_0,\bdt)\\
&=&\Res_{u_0,v_0}\dfunc{y_0}{x_0-z_0}{y_0}\dfunc{u_0}{v_0-z_0}{u_0}\left(\frac{v_0-z_0}{u_0}\right)^{\frac{\lambda}{N}}
    \dfunc{t_0}{u_0-y_0}{t_0}\left( \frac{u_0-y_0}{t_0} \right)^{\frac{\lambda+\mu}{N}}S.
\end{eqnarray*}
Note that for any $\alpha\in\C$, $(x+y)^\alpha$  involves only integer powers of $y$.
Then we have
\begin{eqnarray*}
&&x_0^ky_0^kz_0^k\dfunc{z_0}{x_0-y_0}{z_0}\Ye{a(t_0,\bdt)}{x_0}{\y\z}\Ye{b(t_0,\bdt)}{y_0}{\y}c(t_0,\bdt)\\
&&\quad-x_0^ky_0^kz_0^k\dfunc{z_0}{y_0-x_0}{-z_0}\Ye{b(t_0,\bdt)}{y_0}{\y}\Ye{a(t_0,\bdt)}{x_0}{\y\z}c(t_0,\bdt)\\
&=&\Res_{u_0,v_0}\dfunc{y_0}{x_0-z_0}{y_0}\dfunc{t_0}{v_0-x_0}{t_0}\left( \frac{v_0-x_0}{t_0} \right)^{\frac{\lambda}{N}}
    \dfunc{t_0}{u_0-y_0}{t_0}\left(\frac{u_0-y_0}{t_0}\right)^{\frac{\mu}{N}}S\\
&=&\Res_{u_0,v_0}\dfunc{x_0}{y_0+z_0}{x_0}\dfunc{v_0}{t_0+x_0}{v_0}\left( \frac{t_0+x_0}{v_0} \right)^{-\frac{\lambda}{N}}
    \dfunc{u_0}{t_0+y_0}{u_0}\left(\frac{t_0+y_0}{u_0}\right)^{-\frac{\mu}{N}}S\\
&=&\Res_{u_0,v_0}\dfunc{x_0}{y_0+z_0}{x_0}\dfunc{v_0}{t_0+y_0+z_0}{v_0}\left( \frac{t_0+y_0+z_0}{v_0} \right)^{-\frac{\lambda}{N}}
        \left( \frac{t_0+y_0}{u_0} \right)^{\frac{\lambda}{N}}\\
&&\quad\quad\cdot   \dfunc{u_0}{t_0+y_0}{u_0}\left(\frac{t_0+y_0}{u_0}\right)^{-\frac{\lambda+\mu}{N}}S\\
&=&\Res_{u_0,v_0}\dfunc{x_0}{y_0+z_0}{x_0}\dfunc{v_0}{t_0+y_0+z_0}{v_0}\left( \frac{1+z_0(t_0+y_0)^{-1}}{v_0} \right)^{-\frac{\lambda}{N}}
         u_0 ^{-\frac{\lambda}{N}}\\
&&\quad\quad\cdot   \dfunc{u_0}{t_0+y_0}{u_0}\left(\frac{t_0+y_0}{u_0}\right)^{-\frac{\lambda+\mu}{N}}S\\
&=&\Res_{u_0,v_0}\dfunc{x_0}{y_0+z_0}{x_0}\dfunc{v_0}{u_0+z_0}{v_0}\left( \frac{u_0+z_0}{v_0} \right)^{-\frac{\lambda}{N}}
    \dfunc{u_0}{t_0+y_0}{u_0}\left(\frac{t_0+y_0}{u_0}\right)^{-\frac{\lambda+\mu}{N}}S\\
&=&\Res_{u_0,v_0}\dfunc{y_0}{x_0-z_0}{y_0}\dfunc{u_0}{v_0-z_0}{u_0}\left(\frac{v_0-z_0}{u_0}\right)^{\frac{\lambda}{N}}
    \dfunc{t_0}{u_0-y_0}{t_0}\left( \frac{u_0-y_0}{t_0} \right)^{\frac{\lambda+\mu}{N}}S\\
&=&x_0^ky_0^kz_0^k\dfunc{y_0}{x_0-z_0}{y_0}\Ye{\Ye{a(t_0,\bdt)}{z_0}{\z}b(t_0,\bdt)}{y_0}{\y}c(t_0,\bdt).
\end{eqnarray*}
Thus
\begin{eqnarray*}
&&\dfunc{z_0}{x_0-y_0}{z_0}\Ye{a(t_0,\bdt)}{x_0}{\y\z}\Ye{b(t_0,\bdt)}{y_0}{\y}c(t_0,\bdt)\\
&&\quad\quad-\dfunc{z_0}{y_0-x_0}{-z_0}\Ye{b(t_0,\bdt)}{y_0}{\y}\Ye{a(t_0,\bdt)}{x_0}{\y\z}c(t_0,\bdt)\\
&=&\dfunc{x_0}{y_0+z_0}{x_0}\Ye{\Ye{a(t_0,\bdt)}{z_0}{\z}b(t_0,\bdt)}{y_0}{\y}c(t_0,\bdt).
\end{eqnarray*}
This proves that the Jacobi identity holds. Therefore, $(V,Y_\mathcal{E},1_W)$ carries the structure of an \rtva{}.

As $V$ is a graded subspace of $\mathcal{E}(W,r;N)$, $\sigma$ is naturally a linear automorphism of $V$.
It follows from Lemma \ref{rem:Operators} that
$\sigma$ is an automorphism of $V$ viewed as an \rtva{} and $\sigma^{N}=1$.

For $a(x_0,\x)\in V$, we set
\begin{eqnarray}\label{eq:defofYW123}
\Ymod{W}{a(x_0,\x)}{z_0}{\z}=a(z_0,\z).
\end{eqnarray}
For $a(x_0,\x),b(x_0,\x)\in V$, from assumption
 they are mutually local.
To show that $(W,Y_W)$ is a $\sigma$-twisted $V$-module, in view of Lemma \ref{lem:JacobianComAsso}
we need to establish weak associativity.
Let $k$ be a nonnegative integer such that
$$(x_0-y_0)^k a(x_0,\x)b(y_0,\y)=(x_0-y_0)^k b(y_0,\y)a(x_0,\x).$$
Assume $a(x_0,\x)\in \mathcal{E}(W,r;N)_s$ with $0\leq s<N$.
Let $w\in W$ be an arbitrarily fixed vector in $W$ and let $\ell$ be a nonnegative integer such that
$$x_0^{\ell+\frac{s}{N}}a(x_0,\x)w\in W[[x_0,\x,\x^{-1}]].$$
From Remark \ref{rdifferent-def},  we have
\begin{eqnarray*}
&&z_0^{k}\left(y_0+z_0 \right)^{\ell+\frac{s}{N}}\left(\Ye{a(y_0,\y)}{z_0}{\z}b(y_0,\y)\right)w\\
&=&\left[(x_0-y_0)^{k}x_0^{\ell+\frac{s}{N}}a(x_0,\z\y)b(y_0,\y)w\right]|_{x_0=y_0+z_0}\\
&=&\left[(x_0-y_0)^{k}x_0^{\ell+\frac{s}{N}}a(x_0,\z\y)b(y_0,\y)w\right]|_{x_0=z_0+y_0}\\
&=&z_0^k(z_0+y_0)^{\ell+\frac{s}{N}}a(z_0+y_0,\y\z)b(y_0,\y)w.
\end{eqnarray*}
This proves
\begin{eqnarray*}
&&\left(y_0+z_0 \right)^{\ell+\frac{s}{N}}\Ymod{W}{\Ye{a(t_0,\bdt)}{z_0}{\z}b(t_0,\bdt)}{y_0}{\y}\\
&=&\left(y_0+z_0 \right)^{\ell+\frac{s}{N}}\Ye{a(y_0,\y)}{z_0}{\z}b(y_0,\y)\\
&=&(z_0+y_0)^{\ell+\frac{s}{N}}a(z_0+y_0,\y\z)b(y_0,\y)\\
&=&(z_0+y_0)^{\ell+\frac{s}{N}}\Ymod{W}{a(t_0,\bdt)}{z_0+y_0}{\y\z}\Ymod{W}{b(t_0,\bdt)}{y_0}{\y}.
\end{eqnarray*}
By Lemma \ref{lem:JacobianComAsso}, $W$ is a $\sigma$-twisted $V$-module.
It is clear that $W$ is faithful.
\end{proof}

Using Lemma \ref{lem:localclosed} and Theorem \ref{thm:ConceptionalConstruction},
we  immediately have (cf. \cite{Li2}):

\begin{coro}\label{cgenerated-algebra}
Let $U$ be any graded local subspace of $\mathcal{E}(W,r;N)$. Then there exists a smallest closed graded local subspace
$\<U\>$ that contains $U$, and $\<U\>$ is an $(r+1)$-toroidal vertex algebra
with $W$ as a canonical faithful $\sigma$-twisted module.
\end{coro}

In view of Theorem \ref{thm:ConceptionalConstruction}, we alternatively call a closed graded local subspace
of $\mathcal{E}(W,r;N)$ an {\em $(r+1)$-toroidal vertex subalgebra of $\mathcal{E}(W,r;N)$.}

Let $V$ be any $(r+1)$-toroidal vertex subalgebra of $\mathcal{E}(W,r;N)$.
Note that by Theorem \ref{thm:ConceptionalConstruction}, $(W,Y_W)$ is a faithful $\sigma$-twisted $V$-module
with $Y_{W}(a(x_0,\x);z_0,\z)=a(z_0,\z)$ for $a(x_0,\x)\in V$.
In view of Lemma \ref{ltwistedm-rva}, we immediately have:

\begin{coro}\label{prop:calculationProp0001}
Let $V$ be an $(r+1)$-toroidal vertex subalgebra of $\mathcal{E}(W,r;N)$, let
\begin{eqnarray*}
a(x_0,\x),b(x_0,\x),c_0(x_0,\x),c_1(x_0,\x),\dots,c_k(x_0,\x)\in V,
\end{eqnarray*}
and let $\m\in \Zr$. Assume $a(x_0,\x)\in \mathcal{E}(W,r;N)_s$ with $0\le s<N$. If
\begin{eqnarray}
\left[ a(x_0,\m),b(y_0,\y) \right]=\y^\m\sum\limits_{j=0}^k
c_j(y_0,\y)\frac{1}{j!}\left(\frac{\partial}{\partial y_0} \right)^j
\left(\dfunc{x_0}{y_0}{x_0}\left(\frac{y_0}{x_0}\right)^{\frac{s}{N}} \right),
\end{eqnarray}
where $a(x_0,\x)=\sum\limits_{\n\in\Zr}a(x_0,\n)\x^{-\n}$, then
\begin{eqnarray}
a(x_0,\x)_{j,\m}b(x_0,\x)=c_j(x_0,\x)
\end{eqnarray}
for $0\leq j\leq k$ and $a(x_0,\x)_{j,\m}b(x_0,\x)=0$ for $j>k$. Furthermore, we have
\begin{eqnarray}\label{eq:bracket}
&&\left[ \Ye{a(t_0,\bdt)}{x_0}{\m},\Ye{b(t_0,\bdt)}{y_0}{\y} \right]\nonumber\\
&=&\y^\m \sum\limits_{j=0}^k \Ye{c_j(t_0,\bdt)}{y_0}{\y}\frac{1}{j!}
\left(\frac{\partial}{\partial y_0} \right)^j \dfunc{x_0}{y_0}{x_0}.
\end{eqnarray}
\end{coro}

\section{$(r+1)$-toroidal vertex algebras and twisted modules associated to
twisted toroidal Lie algebras}

In this section, we shall associate \rtva{}s and their
twisted modules to certain twisted toroidal Lie algebras by using the general
construction established in Section \ref{Section:ConceptionalConstruction}.

Let $\g$ be a finite dimensional simple Lie algebra and let $\<\cdot,\cdot\>$ be the killing form
which is normalized such that $\<\alpha,\alpha\>=2$ for any long root $\alpha$.
Let $\sigma_0,\sigma_1,\dots,\sigma_r$ be $r+1$ mutually commuting finite order automorphisms of $\g$,
which are fixed throughout this section.

Set $o(\sigma_i)=N_i$ for $0\leq i\leq r$ and
set
\begin{eqnarray}
G&=&\<\sigma_{0},\sigma_1,\dots,\sigma_r\>\subset \Aut(\g),\\
G_{+}&=&\<\sigma_{1},\sigma_2,\dots,\sigma_r\>\subset G.
\end{eqnarray}
Note that every automorphism of $\g$ preserves the killing form (cf. \cite{K}) and symmetric invariant bilinear forms on $\g$
are unique up to scalar multiples (as $\g$ is simple).
It follows that all $\sigma_i$ for $0\leq i\leq r$ preserve $\<\cdot,\cdot\>$.

Consider the following $(r+1)$-loop Lie algebra
\begin{eqnarray*}
L_{r+1}(\g,N_{0})=\g\otimes\C[\varz{t}{\ovN},\varr{t}].
\end{eqnarray*}
Form a $1$-dimensional central extension
\begin{eqnarray}
\widehat{L_{r+1}}(\g,N_{0})=L_{r+1}(\g,N_{0})\oplus \C \cent,
\end{eqnarray}
where $\cent$ is central and
\begin{eqnarray}
\left[ a\otimes t_0^{m_0'}\bdt^\m,b\otimes t_0^{n_0'}\bdt^\n \right]
=\left[a,b\right]\otimes t_0^{m_0'+n_0'}\bdt^{\m+\n}+
    m_0'\<a,b\>\delta_{m_0'+n_0',0}\delta_{\m+\n,0}\cent
\end{eqnarray}
for $a,b\in\g$, $m_0',n_0'\in \ovN\Z$, $\m,\n\in \Zr$.

\begin{de}
{\em For $0\leq i\leq r$,  we define  an automorphism $\hat{\sigma}_i$ of 
the Lie algebra $\widehat{L_{r+1}}(\g,N_{0})$ by
\begin{eqnarray}
\hat{\sigma}_i(\cent)=\cent, \   \   \
\hat{\sigma}_i(a\otimes t_0^{\frac{m_0}{N_{0}}}\bdt^\m)=\rtu{N_i}^{-m_i}
\left(\sigma_i(a)\otimes t_0^{\frac{m_0}{N_{0}}}\bdt^\m\right)
\end{eqnarray}
for $a\in \g$, $(m_0,\m)\in \Z\times \Z^r$ with $\m=(m_{1},\dots,m_{r})$.}
\end{de}

Note that $\hat{\sigma}_{i}$ depends on both the index $i$ and the corresponding automorphism $\sigma_i$.
We particularly mention that for $0\le i\ne j\le r$, $\hat{\sigma}_i$ and $\hat{\sigma}_j$ 
are different  even though $\sigma_i$ and $\sigma_j$ could be the same. 
The notation $\hat{\sigma}_{i}$, which is somewhat misleading, 
is for the purpose of convenience. 

Set
\begin{eqnarray}
&&\widehat{G}=\<\hat{\sigma}_0,\hat{\sigma}_1,\dots,\hat{\sigma}_r\>\subset \Aut \left(\widehat{L_{r+1}}(\g,N_{0})\right),\\
&&\widehat{G}_{+}=\<\hat{\sigma}_1,\dots,\hat{\sigma}_r\>\subset \widehat{G}.
\end{eqnarray}
It is straightforawrd to show that 
$$\widehat{G}\simeq \<\hat{\sigma}_0\>\times \<\hat{\sigma}_1\>\times \cdots \times \<\hat{\sigma}_r\>
\   \mbox{ and }\  \   
\widehat{G}_{+}\simeq \<\hat{\sigma}_1\>\times \<\hat{\sigma}_2\>\times \cdots \times \<\hat{\sigma}_r\>.$$
Set $\Nfactor=(N_1,\dots,N_r)$. 
Furthermore, we set
\begin{eqnarray}
\Lambda(\Nfactor)=\Z N_1\times \cdots\times \Z N_r\subset \Z^r.
\end{eqnarray}
Then
\begin{eqnarray}
\Z^r/\Lambda(\Nfactor)\simeq \Z_{N_1}\times \cdots\times \Z_{N_r}
\simeq \<\hat{\sigma}_1\>\times \<\hat{\sigma}_2\>\times \cdots \times \<\hat{\sigma}_r\>=\widehat{G}_{+}.
\end{eqnarray}
Note that $G$ is a homomorphism image of $\widehat{G}$ with $\hat{\sigma}_{i}$ 
corresponding to $\sigma_i$ for $0\le i\le r$.
Then $\widehat{G}$ (and $\widehat{G}_{+}$) naturally acts on $\g$ by automorphisms.

Set
\begin{eqnarray}
\tau=\left(\widehat{L_{r+1}}(\g,N_{0})\right)^{\widehat{G}},
\end{eqnarray}
the subalgebra of $\widehat{G}$-fixed points in $\widehat{L_{r+1}}(\g,N_{0})$. Then
\begin{eqnarray}
\Lie=\left(\mathop{\bigoplus}\limits_{(m_0,\m)\in \Z\times \Z^r}\g_{\suball{m_0}{\m}}\otimes t_0^{\frac{m_0}{N_{0}}}\bdt^\m\right)\oplus \C\cent,
\end{eqnarray}
where for $m_0\in \Z$, $\m=(m_{1},\dots,m_{r})\in \Z^{r}$,
\begin{eqnarray}
\g_{\suball{m_0}{\m}}=\set{ a\in \g }{\sigma_i(a)=\rtu{N_i}^{m_i}a\  \,\fortext
0\leq i\leq r }.
\end{eqnarray}

For $a\in \g_{\suball{k_0}{\rsymbol{k}}}$ with $(k_0,\rsymbol{k})\in \Z\times \Z^r$, set
\begin{eqnarray}
a^{\tau}(x_0,\x)=\sum_{m_0\in \Z,\  \m\in \rsymbol{k}+ \Lambda(\Nfactor)}
\left(a\otimes t_0^{\frac{k_0}{N_{0}}+m_0}\bdt^{\m}\right)x_0^{-\frac{k_0}{N_{0}}-m_0-1}\x^{-\m},
\end{eqnarray}
and we also set
\begin{eqnarray}
a^{\tau}(x_0,\m)=\sum_{m_0\in \Z}\left(a\otimes t_0^{\frac{k_0}{N_{0}}+m_0}\bdt^{\m}\right)
x_0^{-\frac{k_0}{N_{0}}-m_0-1}
\end{eqnarray}
for $\m\in\rsymbol{k}+ \Lambda(\Nfactor)$, and define $a^{\tau}(x_0,\m)=0$
for $\m\notin\rsymbol{k}+ \Lambda(\Nfactor)$.
 We then define $a^{\tau}(x_0,\x)$ for general $a\in \g$ by linearity.

For $a\in \g_{\suball{k_0}{\rsymbol{k}}},\  b\in\g_{\suball{\ell_0}{\rsymbol{l}}}$,
$\m\in\rsymbol{k}+ \Lambda(\Nfactor)\subset \Z^r$, we have
\begin{eqnarray}\label{eq:commutator}
\left[ a^{\tau}(x_0,\m),b^{\tau}(y_0,\y)
\right]&=&\y^{\m}[a,b]^{\tau}(y_0,\y)\dfunc{x_0}{y_0}{x_0}\left(\frac{y_0}{x_0}
\right)^{\frac{k_0}{N_{0}}}\nonumber\\
&& +\y^{\m}\<a,b\> \cent \frac{\partial}{\partial y_0}
    \left( \dfunc{x_0}{y_0}{x_0}\left(\frac{y_0}{x_0} \right)^{\frac{k_0}{N_{0}}}
    \right).
\end{eqnarray}
Note that for $\m\notin\rsymbol{k}+ \Lambda(\Nfactor)$, we have
$\left[ a^{\tau}(x_0,\m),b^{\tau}(y_0,\y) \right]=0$.

Let $\ell\in \C$.  A $\Lie$-module $W$ is said to be of {\em level} $\ell$
if $\cent$ acts on $W$ as scalar $\ell$, and $W$ is called a {\em restricted module} if
\begin{eqnarray}
a^{\tau}(x_0,\x)\in\ESalltwisted{W}\quad\text{for all }a\in\g.
\end{eqnarray}

Let $W$ be a restricted $\Lie$-module of level $\ell$. Set
\begin{eqnarray}
U_W=\Span\set{a^{\tau}(x_0,\x)}{a\in\g}.
\end{eqnarray}
It follows from (\ref{eq:commutator}) that $U_W$ is a graded local subspace of $\ESalltwisted{W}$.
In view of Corollary \ref{cgenerated-algebra},  $U_W$ generates
 an $(r+1)$-toroidal vertex subalgebra  $\<U_W\>$ of  $\ESalltwisted{W}$ and
$W$ is a faithful $\sigma$-twisted $\<U_W\>$-module with
$\Ymod{W}{\alpha(x_0,\x)}{z_0}{\z}=\alpha(z_0,\z)$ for $\alpha(x_0,\x)\in \<U_W\>$.
In the following, we are going to characterize $\<U_W\>$ as a module for another Lie algebra.

Consider the $(r+1)$-loop algebra 
$$L_{r+1}(\g)=\g\otimes \C[t_0^{\pm 1},t_1^{\pm 1},\dots,t_r^{\pm 1}],$$
which is equal to $L_{r+1}(\g,N_0)$ with $N_0=1$. As before, $\widehat{G}_{+}$ acts on $\widehat{L_{r+1}}(\g)$ by automorphisms.
Set
\begin{eqnarray}
\Lieu=\left(\widehat{L_{r+1}}(\g)\right)^{\widehat{G}_{+}},
\end{eqnarray}
the subalgebra of $\widehat{G}_{+}$-fixed points in $\widehat{L_{r+1}}(\g)$.
Then
\begin{eqnarray}
\mathcal{L}=\left(\mathop{\bigoplus}\limits_{\m\in \Zr}\g_\m\otimes  \bdt^\m\C\left[
t_0^{\pm 1} \right]\right) \bigoplus \C\cent,
\end{eqnarray}
where for $\m=(m_{1},\dots,m_{r})\in \Z^r$,
\begin{eqnarray*}
\g_{\m}=\set{a\in\g}{\sigma_j(a)=\rtu{N_j}^{m_j}a\ \mbox{ for  } 1\leq j\leq r }.
\end{eqnarray*}


Set
\begin{eqnarray*}
\Lieu^{\ge 0}&=&\left(\mathop{\bigoplus}\limits_{\m\in\Zr}\g_\m\otimes
\bdt^\m\C\left[t_0 \right] \right)\oplus \C\cent,\nonumber\\
\Lieu^{-}&=&\mathop{\bigoplus}\limits_{\m\in\Zr}\g_\m\otimes
\bdt^\m t_0^{-1}\C \left[t_0^{-1} \right],
\end{eqnarray*}
which are subalgebras of $\Lieu$.
We have
\begin{eqnarray*}
\Lieu=\Lieu^{\ge 0}\oplus \Lieu^{-}.
\end{eqnarray*}

Let $\chi: \widehat{G}_{+}\rightarrow \C^{\times}$ be the group homomorphism, i.e., 
a linear character of $ \widehat{G}_{+}$, uniquely determined by
\begin{eqnarray*}
\chi (\hat{\sigma}_i)= \omega_{N_i} \   \   \  \mbox{ for } 1\le i\le r.
\end{eqnarray*}
More generally, for any $\m=(m_1,\dots,m_r)\in \Z^r$, we define $\chi^{\m}$ to be the linear character of $\widehat{G}_{+}$, 
uniquely determined by 
\begin{eqnarray}
\chi^{\m} (\hat{\sigma}_i)= \omega_{N_i}^{m_i} \   \   \  \mbox{ for } 1\le i\le r.
\end{eqnarray}
It can be readily seen that for $\m,\n\in \Z^r$, $\chi^{\m}=\chi^{\n}$ if and only if $\m-\n \in \Lambda({\bf N})$.
For $\n\in \Z^r$, using the character $\chi^{\n}$ we have
 \begin{eqnarray}
 \g_{\n}=\{ a\in \g\ |\ \gamma (a)=\chi^{\n}(\gamma)a\   \   \    \mbox{ for }\gamma\in \widehat{G}_{+}\}.
 \end{eqnarray}
It follows that for $\m,\n\in \Z^r$ with $\m-\n\in \Lambda({\bf N})$, we have $\g_{\m}=\g_{\n}$. 
For any $[\n]\in \Z^r/\Lambda({\bf N})$ with $\n\in \Z^r$, we define $\g_{[\n]}=\g_{\n}$. Then
\begin{eqnarray}
\g=\bigoplus_{[\n]\in \Z^r/\Lambda({\bf N})} \g_{[\n]}.
\end{eqnarray}

For $a\in \g$ and $\m\in \Z^r$, set
\begin{eqnarray}
a_{(\m)}=\frac{1}{N_{+}}\sum_{\gamma\in \widehat{G}_{+}}\chi^{\m}(\gamma^{-1})\gamma(a),
\end{eqnarray}
where $N_{+}=N_1N_2\cdots N_r=|\widehat{G}_{+}|$.
For $a\in \g_{\n}$ with $\n\in \Z^r$, we have $a_{(\m)}=a$ for $\m\in \n+\Lambda({\bf N})$ and
$a_{(\m)}=0$ for $\m\notin \n+\Lambda({\bf N})$.
Then $\g_{\m}=\{ a_{(\m)}\  |\  a\in \g\}$. 

For  $a\in \g,\ (m_0,\m)\in \Z\times \Z^r$, set
\begin{eqnarray}
a^{\Lieu}(m_0,\m)=a_{(\m)}\otimes t_0^{m_0}\bdt^{\m}\in \Lieu.
\end{eqnarray}
Furthermore, for $a\in \g$, set
\begin{eqnarray}
a^{\Lieu}(x_0,\x)&=&\sum_{(m_0,\m)\in \Z\times \Z^r}a^{\Lieu}(m_0,\m)x_0^{-m_0-1}\x^{-\m}\nonumber\\
&=&\sum_{(m_0,\m)\in \Z\times \Z^r}(a_{(\m)}\otimes t_0^{m_0}\bdt^{\m})x_0^{-m_0-1}\x^{-\m}.
\end{eqnarray}
In particular, if $a\in \g_{\rsymbol{k}}$ with $\rsymbol{k}\in \Zr$, we have
\begin{eqnarray}
a^{\Lieu}(x_0,\x)&=&\sum_{m_0\in \Z,\m\in \rsymbol{k}+\Lambda(\Nfactor)}
\left(a\otimes t_0^{m_0}\bdt^{\m}\right)x_0^{-m_0-1}\x^{-\m}.
\end{eqnarray}
Let $a\in \g_\rsymbol{k}$, $b\in \g_\rsymbol{l}$ with $\rsymbol{k}, \rsymbol{l}\in \Z^r$.
Then
\begin{eqnarray}\label{eq:commutatorUntwisted}
&&\left[a^{\Lieu}(x_0,\m),b^{\Lieu}(y_0,\y)
\right]=\y^{\m}\left([a,b]^{\Lieu}(y_0,\y)\dfunc{x_0}{y_0}{x_0}
+\<a,b\> \cent \frac{\partial}{\partial y_0}
     \dfunc{x_0}{y_0}{x_0}\right)  \nonumber\\
     &&
\end{eqnarray}
for $\m\in\rsymbol{k}+\Lambda(\Nfactor)\subset \Zr$. Note that $\left[a^{\Lieu}(x_0,\m),b^{\Lieu}(y_0,\y)
\right]=0$ for $\m\notin\rsymbol{k}+\Lambda(\Nfactor)$.

An $\Lieu$-module $W$ is said to be of {\em level $\ell\in \C$} if $\cent$ acts on $W$ as scalar $\ell$
and $W$ is called a {\em restricted module} if
\begin{eqnarray*}
a^{\Lieu}(x_0,\x)\in\ESall{W}\   \   \mbox{ for  all }a\in\g.
\end{eqnarray*}

We have:

\begin{lem}\label{lem:ActingOnUW}
Let $W$ be any restricted $\tau$-module of level $\ell\in \C$. Then
$\<U_W\>$ is a restricted $\Lieu$-module of level $\ell$ with $a^{\Lieu}(y_0,\y)$ acting as
$\Ye{a^{\tau}(x_0,\x)}{y_0}{\y}$ for $a\in \g$
 and with $\cent$ acting as scalar $\ell$.
Furthermore, $U_{W}+\C 1_{W}$ is an $\Lieu^{\ge 0}$-submodule of $\<U_{W}\>$, where
for $a\in {\g}_{\rsymbol{k}}, b\in \g$, $j\in \N$, $\m\in \rsymbol{k}+\Lambda(\Nfactor)\subset\Zr$,
\begin{eqnarray*}
&&\left(a\otimes t_0^j \bdt^{\m}\right)\cdot1_{W}=a(x_0,\x)^{\tau}_{j,\m} 1_{W}=0,\nonumber\\
&&\left(a\otimes t_0^j \bdt^{\m}\right)\cdot b^{\tau}(x_0,\x)=a^{\tau}(x_0,\x)_{j,\m} b^{\tau}(x_0,\x)
=\left\{\begin{array}{ll} [a,b]^{\tau}(x_0,\x) & \text{if }j=0\\ \<a,b\>\ell 1_W &\text{if }j=1\\ 0 & \text{if }j\geq 2.   \end{array}\right.
\end{eqnarray*}
\end{lem}

 \begin{proof} Let $a\in {\g}_{\rsymbol{k}}, b\in \g$ with $\rsymbol{k}\in \Z^r$ and let
$\m\in\rsymbol{k}+\Lambda(\Nfactor)\subset \Zr$.
 Combining (\ref{eq:commutator}) with Corollary \ref{prop:calculationProp0001}, we get
\begin{eqnarray}
&&\left[ \Ye{a^{\tau}(t_0,\bdt)}{x_0}{\m},\Ye{b^{\tau}(t_0,\bdt)}{y_0}{\y} \right]\nonumber\\
&=&\y^{\m}\left(
    \Ye{\left[a,b\right]^{\tau}(t_0,\bdt)}{y_0}{\y}\dfunc{x_0}{y_0}{x_0}+
    \<a,b\>\ell\frac{\partial}{\partial y_0}  \dfunc{x_0}{y_0}{x_0}
    \right).
\end{eqnarray}
On the other hand, notice that from (\ref{eq:defofYe}), we have
$\Ye{a^{\tau}(t_0,\bdt)}{x_0}{\m}=0$ for $\m\notin \rsymbol{k}+\Lambda(\Nfactor)$.
Then the first assertion follows. From Corollary \ref{prop:calculationProp0001}, the second assertion also follows.
 \end{proof}




Notice that $\widehat{L_{r+1}}(\g)$ is a $\Z$-graded Lie algebra
with $\deg (a\otimes t_0^k\bdt^\m)=-k$ for $k\in \Z,\ \m\in \Zr$ and $\deg \cent=0$,
where  $\Lieu$, $\Lieu^{\ge 0}$, and $\Lieu^{-}$ are all $\Z$-graded subalgebras.
From Lemma 4.1 of \cite{LTW}, we immediately have:

\begin{lem}\label{lem:BorelMod}
Let $\ell$ be a complex number. Then there exists an $\Lieu^{\ge 0}$-module structure on $\g\oplus \C$
with $\cent$ acting as scalar $\ell$ and with
\begin{eqnarray}
&&\left(a\otimes t_0^j \bdt^{\m}\right)\cdot \C=0,\nonumber\\
&&\left(a\otimes t_0^j \bdt^{\m}\right)\cdot b
=\left\{\begin{array}{ll} [a,b] & \text{if }j=0\\ \<a,b\>\ell &\text{if }j=1\\ 0 & \text{if }j\geq 2    \end{array}\right.
\end{eqnarray}
for $a\in\g_{\rsymbol{k}},\ b\in \g$, $j\in \N$, $\m\in \rsymbol{k}+\Lambda(\Nfactor)\subset \Zr$.
Furthermore, $\g\oplus\C$ is an $\N$-graded $\Lieu^{\ge 0}$-module with $\deg \C=0$ and $\deg \g=1$.
\end{lem}

For $\ell\in \C$, denote by $\left(\g\oplus \C\right)_\ell$ the $\Lieu^{\ge 0}$-module
obtained in Lemma \ref{lem:BorelMod}. Then form an induced module
\begin{eqnarray}
\T=U\left( \Lieu \right)\otimes_{U\left(\Lieu^{\ge 0}\right)} \left(\g\oplus \C\right)_\ell.
\end{eqnarray}
In view of the P-B-W theorem, we have
\begin{eqnarray*}
\T=U\left( \Lieu^{-} \right)\otimes \left(\g\oplus \C\right)
\end{eqnarray*}
as a vector space. Consequently, $\T$ is an $\N$-graded $\Lieu$-module.
It follows that $\T$ is a restricted $\Lieu$-module of level $\ell$.
Set
\begin{eqnarray}
\vac=1\otimes 1\in \T.
\end{eqnarray}
Identify $\g$ with the subspace $1\otimes \g$ through the map $a\mapsto 1\otimes a$.

\begin{coro}\label{cmodule-homomorphism}
Let $W$ be a restricted $\Lie$-module of level $\ell$. Then  there exists
an  $\Lieu$-module homomorphism
\begin{eqnarray}\label{eq:LieUntwistedModHom}
\phi_{W}: \  \T\rightarrow \<U_W\>\subset \ESalltwisted{W}
\end{eqnarray}
such that
$$\phi_{W}(a+\mu)=a^{\tau}(x_0,\x)+\mu 1_{W}\   \   \mbox{ for }a\in\g,\ \mu\in \C.$$
\end{coro}

\begin{proof} It  follows from Lemma \ref{lem:ActingOnUW} that there exists
an $\Lieu^{\ge 0}$-module  homomorphism $\phi_{W}^{0}$
from $\left(\g\oplus\C\right)_\ell$ to $\<U_W\>$ such that
$$\phi_{W}^{0}(a+\mu)=a^{\tau}(x_0,\x)+\mu 1_{W}\   \   \mbox{ for }a\in\g,\ \mu\in \C.$$
Then from the construction of $\T$, we see that
 $\phi_{W}^{0}$ extends uniquely to an $\Lieu$-module homomorphism
 $\phi_{W}: \  \T\rightarrow \<U_W\>\subset \ESalltwisted{W}$, as desired.
\end{proof}

We have:

\begin{thm}
Let $\ell$ be any complex number. Then there exists an \rtva{} structure on $\T$, which is uniquely determined by
the conditions that $\vac$ is the vacuum vector and that
\begin{eqnarray}\label{eq:TVAaction}
\Y{a}{x_0}{\x}=a^{\Lieu}(x_0,\x)\quad\fortext a\in \g.
\end{eqnarray}
\end{thm}

\begin{proof}
Since $\T=U\left( \Lieu \right)\left(\g+\C\vac\right)$, the uniqueness follows immediately.
In the following, we establish the existence by applying Theorem 3.10 in \cite{LTW}.
Temporarily, take $\sigma_0=1$ (with $N_{0}=1$) and $W=\T$ in (\ref{eq:LieUntwistedModHom}).
Then we have $\tau=\Lieu$ and
\begin{eqnarray*}
a^{\tau}(x_0,\x)=a^{\Lieu}(x_0,\x)\in\ESall{\T}\  \  \mbox{ for }a\in\g.
\end{eqnarray*}
From Corollary \ref{cmodule-homomorphism}, we have an $\Lieu$-homomorphism
$\phi_{W}$ from $\T$ to $\<U_{W}\>$, which we alternatively denote by $\phi_{x_0,\x}$.
Define
\begin{eqnarray}
\Y{\cdot}{x_0}{\x}:&&\T\rightarrow \ESall{\T}\nonumber\\
                && v\mapsto \phi_{x_0,\x}(v).
\end{eqnarray}
For $a\in\g$, we have
\begin{eqnarray*}
\Y{a}{x_0}{\x}=\phi_{x_0,\x}(a)=a^{\Lieu}(x_0,\x).
\end{eqnarray*}
Then $\set{\Y{a}{x_0}{\x}}{a\in \g}\cup \{1_{W}\}$ is local. For $a\in \g$, $v\in \T$, we have
\begin{eqnarray*}
&&\Y{\Y{a}{z_0}{\z}v}{y_0}{\y}\\
&=&\phi_{y_0,\y}\left( a^{\Lieu}(z_0,\z)v \right)\\
&=&\Ye{a^{\Lieu}(y_0,\y)}{z_0}{\z}\phi_{y_0,\y}(v)\\
&=&\Res_{x_0}\left(\dfunc{z_0}{x_0-y_0}{z_0}a^{\Lieu}(x_0,\z\y)\phi_{y_0,\y}(v)-\dfunc{z_0}{y_0-x_0}{-z_0}\phi_{y_0,\y}(v)a^{\Lieu}(x_0,\z\y)\right)\\
&=&\Res_{x_0}(\dfunc{z_0}{x_0-y_0}{z_0}\Y{a}{x_0}{\z\y}\Y{v}{y_0}{\y}\\
&&\quad  \quad
-\dfunc{z_0}{y_0-x_0}{-z_0}\Y{v}{y_0}{\y}\Y{a}{x_0}{\z\y}).
\end{eqnarray*}
It then follows fromTheorem 3.10 in \cite{LTW} that $\left( \T,Y,\vac \right)$ carries the structure of
an \rtva{}.
\end{proof}

Recall that $\g\oplus\C\subset \T$ and
$\sigma_{0},\dots,\sigma_r$ are automorphisms of $\g$. 

\begin{lem}\label{lauto-sigma0}
For each $0\le i\le r$, there is an automorphism $\tilde{\sigma}_{i}$ of $\T$, 
which extends the automorphism $\sigma_{i}$ of $\g$ uniquely.
\end{lem}

\begin{proof} Let $i\in \{ 0,1,\dots,r\}$.
As $\T$  as an \rtva{} is generated by $\g$ and $\mathbf{1}$, the uniqueness is clear.
It remains to establish the existence. 
Since the automorphism $\sigma_i$ of $\g$ preserves the bilinear form,
we can and we should view $\sigma_{i}$ as an automorphism of the Lie algebra $\widehat{L_{r+1}}(\g)$ by
\begin{eqnarray*}
\sigma_{i}(a\otimes t_0^{m_0}\bdt^\m+\mu\cent)=\sigma_i(a)\otimes
t_0^{m_0}\bdt^\m+\mu\cent
\end{eqnarray*}
for $a\in{\g},\  (m_{0},\m)\in\mathbb{Z}\times {\mathbb{Z}}^{r},\  \mu\in\C$. 
With all $\sigma_{j}$ $(0\leq j\leq r)$ commuting, $\sigma_i$ preserves the subalgebra $\Lieu$.
This induces an automorphism of the universal enveloping algebra $\mathcal{U}(\Lieu)$, 
also denoted by $\sigma_{i}$. It is clear that $\sigma_{i}$ preserves the subalgebra $\Lieu^{\ge 0}$.

Recall from Lemma \ref{lem:BorelMod} the $\Lieu^{\ge 0}$-module structure on $\g\oplus \C$.
Let $\bar{\sigma}_{i}$ be the linear endomorphism of $\g\oplus \C$, defined by
$$\bar{\sigma}_{i}(a+\lambda)=\sigma_i(a)+\lambda\    \    \mbox{ for }a\in \g,\ \lambda\in \C.$$
From the definition of the action of $\Lieu^{\ge 0}$ in Lemma \ref{lem:BorelMod}, we have
$$\bar{\sigma}_{i}(X\cdot v)=\sigma_i(X)\bar{\sigma}_{i}(v)
\   \   \   \mbox{ for }X\in \Lieu^{\ge 0},\ v\in \g\oplus \C.$$
Furthermore, this holds for all $X\in U(\Lieu^{\ge 0})$. By using this property, it is straightforward to show that
the linear automorphism $\sigma_{i}\otimes \bar{\sigma}_{i}$ of $U(\Lieu)\otimes (\g\oplus \C)$ reduces to a
linear automorphism  of $\T$, which we denote by $\tilde{\sigma}_{i}$.
We have $\tilde{\sigma}_{i}(a+\mu\mathbf{1})=\sigma_i(a)+\mu \mathbf{1}$ for $a\in \g,\ \mu\in \C$ and 
$\tilde{\sigma}_{i}(Xv)=\sigma_{i}(X)\tilde{\sigma}_{i}(v)$ for $X\in \mathcal{U}(\Lieu),\  v\in \T.$
 In particular, we have
\begin{eqnarray*}
&&\tilde{\sigma}_{i}(a_{m_{0},\m}v)=
\tilde{\sigma}_{i}(a\otimes t_0^{m_0}\bdt^\m v)=\sigma_{i}(a)\otimes t_0^{m_0}\bdt^\m \tilde{\sigma}(v)
=\sigma_{i}(a)_{m_0,\m}\tilde{\sigma}_{i}(v)
\end{eqnarray*}
for $a\in {\g}_{\m},\  m_{0}\in\mathbb{Z},\  \m\in{\mathbb{Z}}^{r}.$
Since $\T$  as an \rtva{} is generated by $\g$ and $\mathbf{1}$, it follows that
$\tilde{\sigma}_{i}$ is an automorphism of \rtva{} $\T$.
With $\sigma_{i}$ of order $N_{i}$, we see that $\tilde{\sigma}_{i}$ is also of order $N_{i}$.
\end{proof}

Now, we are in a position to present the main result of this paper.

\begin{thm}\label{last-thm}
Let $\ell\in \C$. For any restricted $\Lie$-module $W$ of level $\ell$,
there exists a $\tilde{\sigma}_{0}$-twisted $\T$-module structure on $W$, which is uniquely determined by
\begin{eqnarray*}
\Ymod{W}{a}{x_0}{\x}=a^{\tau}(x_0,\x) \quad \fortext a\in \g
\end{eqnarray*}
and which satisfies the condition that
\begin{eqnarray}\label{eEquivariance-thm}
\Ymod{W}{\sigma_i( v)}{x_0}{\x}=\lim_{x_i\rightarrow \omega_i^{-1}x_i}\Ymod{W}{v}{x_0}{\x}
\end{eqnarray}
for $v\in \T,\ 1\le i\le r$.
On the other hand, for any $\tilde{\sigma}_{0}$-twisted $\T$-module $(W,Y_W)$
which satisfies (\ref{eEquivariance-thm}), $W$ is a restricted $\Lie$-module
of level $\ell$ with
\begin{eqnarray*}
a^{\tau}(x_0,\x)=\Ymod{W}{a}{x_0}{\x}\quad \fortext a\in \g.
\end{eqnarray*}
\end{thm}

\begin{proof}
From Corollary \ref{cgenerated-algebra}, $W$ is a faithful $\sigma$-twisted $\<U_W\>$-module.
To show that $W$ is a $\tilde{\sigma}_{0}$-twisted $\T$-module,  we next show that there exists an \rtva{}
homomorphism $\psi:\T\rightarrow\<U_W\>$ such that $\psi\circ \tilde{\sigma}_{0}=\sigma\circ\psi$.

From Corollary \ref{cmodule-homomorphism}, we have an $\Lieu$-module homomorphism
$\phi_{W}$ from $\T$ to $\<U_{W}\>$. Notice that
\begin{eqnarray*}
\phi_{W}(a)=a^{\tau}(x_0,\x)\quad\fortext a\in\g.
\end{eqnarray*}
Then for any $a\in \g$, $v\in \T$, we have
\begin{eqnarray*}
&&\phi_{W}\left(\Y{a}{z_0}{\z}v\right)=\phi_{W}\left(a^{\Lieu}(z_0,\z)v \right)\\
&&\quad=\Ye{a^{\tau}(x_0,\x)}{z_0}{\z}\phi_{W}(v)=\Ye{\phi_{W}(a)}{z_0}{\z}\phi_{W}(v).
\end{eqnarray*}
From Lemma 2.10 in \cite{LTW}, we see that $\phi_{W}$ is
an \rtva{} homomorphism from $\T$ to $\<U_W\>$.

For $a\in \g_{s,\rsymbol{s}}$ with $(s,{\bf s})\in \Z\times \Z^{r}$, as $a^{\tau}(x_0,\x)\in \ESalltwisted{W}_s$,
we have
\begin{eqnarray*}
&&\phi_{W}(\tilde{\sigma}_{0}(a))=\rtu{N}^s \phi_{W}(a)=\rtu{N}^s
a^{\tau}(x_0,\x) =\sigma(\phi_{W}(a)).
\end{eqnarray*}
Since $\T$ is generated by $\g$ and $\mathbf{1}$, it follows that
$\phi_{W}\circ\tilde{\sigma}_{0}=\sigma\circ \phi_{W}$. Then
$\phi_{W}$ is an \rtva{} homomorphism that we need.
Consequently, $W$ is a $\tilde{\sigma}_{0}$-twisted $\T$-module 
with $Y_{W}(a;x_0,\x)=a^{\tau}(x_0,\x)$ for $a\in \g$. Let $a\in \g_{\m}$ with $\m\in \Z^r$. 
For $1\le i\le i$, we have
\begin{eqnarray}
(\sigma_ia)^{\tau}(x_0,\x)=\omega_i^{m_i}a^{\tau}(x_0,\x)=\lim_{x_i\rightarrow \omega_i^{-1}x_i}a^{\tau}(x_0,\x).
\end{eqnarray}
As $\g+\C$ generates $\T$ as a toroidal vertex algebra, 
using the twisted iterate formula (\ref{eq:iterate}) we obtain (\ref{eEquivariance-thm}).

Conversely, let $(W,Y_W)$ be a $\tilde{\sigma}_{0}$-twisted $\T$-module. From the construction of \rtva{} $\T$,
 we have
\begin{eqnarray*}
&&a_{j,\m}b=\left\{\begin{array}{ll} [a,b] & \text{if }j=0\\
\<a,b\>\ell {\bf 1}&\text{if }j=1\\ 0 & \text{if }j\geq 2    \end{array}\right.
\end{eqnarray*}
 for $a\in {\g}_{\mathbf{s}},\  b\in \g, \ \m\in \mathbf{s}+\Lambda(\Nfactor)
 \subset \Zr$.
Then using (\ref{eq:commutator2}) we obtain
\begin{eqnarray*}
&&\left[ \Ymod{W}{a}{x_0}{\m},\Ymod{W}{b}{y_0}{\y} \right]\\
&=&\y^\m\Ymod{W}{ [a,b]}{y_0}{\y}\dfunc{x_0}{y_0}{x_0}\left(\frac{y_0}{x_0}
\right)^{\frac{s}{N}}+\y^\m\<a,b\>\ell \frac{\partial}{\partial y_0}\left(
    \dfunc{x_0}{y_0}{x_0}\left(\frac{y_0}{x_0} \right)^{\frac{s}{N}}\right)
\end{eqnarray*}
for $\m\in \rsymbol{s}+\Lambda(\Nfactor)$. On the other hand, from (\ref{eEquivariance-thm}) we have
$ \Ymod{W}{a}{x_0}{\m}=0$ for $\m\notin \rsymbol{s}+\Lambda(\Nfactor)$.
Combining these with (\ref{eq:commutator}),
we conclude that $W$ is a restricted $\Lie$-module of level $\ell$ with
$a^{\tau}(x_0,\x)=\Ymod{W}{a}{x_0}{\x}$ for $a\in \g$.
\end{proof}

\end{document}